\documentclass[10pt]{article}
\usepackage[utf8]{inputenc}
\usepackage{amssymb}
\usepackage[T1]{fontenc}
\usepackage{lmodern}
\usepackage[french,english]{babel}
\usepackage{hyphenat}
\usepackage{xspace}
\usepackage{mathrsfs}
\usepackage{amsmath}
\usepackage{amsthm}
\usepackage{tikz}
\usepackage{upgreek}
\usetikzlibrary{matrix,arrows}
\definecolor{gris}{gray}{0.45}
\usepackage{dsfont}
\usepackage{url}
\usepackage{blindtext}
\usepackage{dsfont}
\usepackage[sectionbib,square]{natbib}

\DeclareMathOperator{\im}{\mathrm{Im}}
\DeclareMathOperator{\Spec}{Spec}
\DeclareMathOperator{\Specmax}{Spec_{\fkm}}
\DeclareMathOperator{\Specf}{Spf}

\DeclareMathOperator{\dega}{\widehat{deg}}
\DeclareMathOperator{\rk}{rk}

\DeclareMathOperator{\Sym}{Sym}

\DeclareMathOperator{\pentemax}{\widehat{\mu}_{\max}}

\DeclareMathOperator{\Aut}{Aut}

\DeclareMathOperator{\Card}{Card}
\DeclareMathOperator{\Gal}{Gal}

\def\bA{\mathbf{A}}

\def\bR{\mathbf{R}}
\def\bC{\mathbf{C}}
\def\bZ{\mathbf{Z}}
\def\bN{\mathbf{N}}
\def\bQ{\mathbf{Q}}
\def\bF{\mathbf{F}}
\def\bP{\mathbf{P}}

\def\fkF{\mathfrak{F}}
\def\fkp{\mathfrak{p}}
\def\fkq{\mathfrak{q}}
\def\fko{\mathfrak{o}}
\def\fkm{\mathfrak{m}}

\def\sD{\mathscr D}

\def\sE{\mathscr E}

\def\ol{\overline}
\def\t{\mathsf{t}}

\def\GL{\mathrm{GL}}

\newcommand{\Osheaf}{{\mathscr O}}
\newcommand{\ie}{\textit{i.e. }}

\newcommand{\cf}{\textit{cf.~}}

\newtheorem{theo}{Theorem}[section]
\newtheorem{lemme}[theo]{Lemma}
\newtheorem{prop}[theo]{Proposition}
\newtheorem{definition}[theo]{Definition}

\newtheorem*{theo*}{Theorem}

\theoremstyle{remark}

\newtheorem*{remarque}{Remark}

\numberwithin{equation}{section}

\title{Algebraic points on meromorphic curves}
\author{Mathilde Herblot\footnote{Goethe Universität, Frankfurt am Main - email: mathilde.herblot@gmail.com}}

\begin{document}

\maketitle

\begin{abstract}
The classic Schneider-Lang theorem in transcendence theory asserts that there are only finitely many points at which algebraically independent complex meromorphic functions of finite order of growth can simultaneously take values in a number field, when satisfying a polynomial differential equation with coefficients in this given number field. In this article, we are interested in generalizing this theorem in two directions. First, instead of considering meromorphic functions on~$\bC$ we consider holomorphic maps on an affine curve over the field~$\bC$ or~$\bC_p$. This extends a statement of  D.~Bertrand 
 which applies to meromorphic functions on~$\bP^1(\bC)$ or $\bP^1(\bC_p)$ minus a finite subset of points.  Secondly, we deal with algebraic values taken by the functions, instead of rational values as in the classic setting, inspired by a work of D.~Bertrand. 
 We prove a geometric statement extending those two results, using the \emph{slopes method}, written in the language of Arakelov geometry. In the complex case, we recover a special case of a result by C.~Gasbarri.

\end{abstract}

\section*{Introduction}

Let $f_1,\dots,f_n$ be meromorphic functions on~$\bC$ and assume that at least two of these functions are algebraically independent over~$\bQ$. The Schneider-Lang theorem asserts that the set $W_K$ of points at which the functions $f_1,\dots,f_n$ simultaneously take values in a given number field~$K$ is finite, under two hypotheses. The first condition is that the functions satisfy a polynomial differential equation with coefficients in~$K$; in other words, the ring $K[f_1,\dots,f_n]$ is stable under the derivation $d/dz$. The second one is that the functions $f_1,\dots,f_n$ have a finite order of growth. \label{def ordre de croissance classique}
We recall that an entire function $f$ on~$\bC$ is said to be \emph{of order at most~$\rho$}, for a non-negative real number~$\rho$, if there exist non-negative real numbers $A,B$ such that, for all~$z\in\bC,$
\[|f(z)|\leqslant Ae^{B|z|^{\rho}}.\]
A meromorphic function on~$\bC$ is \emph{of order at most~$\rho$} if it can be written as the quotient of two holomorphic functions of order at most~$\rho.$

We give here a generalization of this statement. Let $K$ be a number field and let~$X$ be a projective variety defined over~$K$. We show that there are only finitely many formal subschemes over~$K$ of dimension~$1$ of the formal completion of~$X$ at a $K$-rational point which satisfy two conditions. In the classic statement, the variety~$X$ would be $\bP^n(\bC)$ and the formal subschemes would be the germs of formal curves defined by $f=(1,f_1,\dots,f_n)$ at the points of $f(W_K)$. The first condition we impose will be called \emph{$\alpha$-arithmeticity} and generalizes the condition of differential equation, and the second condition will be called \emph{uniformization of finite order} of the subschemes and generalizes the hypothesis of germs of curves parameterized by meromorphic function of finite order on~$\bC$.

Moreover, instead of considering formal subschemes based at $K$-rational points for a given number field~$K$, we will consider formal subschemes based at any closed point of~$X$.

\bigskip

To establish our result  about formal subschemes of dimension~1 at the algebraic points we will suppose that they satisfy an ``Arakelovian''  hypothesis: we will say that such a formal subscheme is \emph{$\alpha$-arithmetic} if the height of an \emph{evaluation morphism} along this formal subscheme satisfies some upper bound (see Paragraph~\ref{section morphisme evaluation} for the definition of the evaluation morphisms, and Paragraph~\ref{section alpha arithmetique} for the definition of an $\alpha$-arithmetic formal subscheme). The smaller the non-negative real parameter~$\alpha$ is, the stronger the condition is.

In the case of a polynomial differential equation, that is to say of an algebraic foliation, a germ of leaf at an algebraic point defines a $1$-arithmetic formal subscheme (Lemmas~\ref{equa diff implique 1-analytique} and~\ref{alpha analytique arithmetique}). If moreover almost all $\fkp$-curvatures of the foliation vanish, the formal subscheme is 0-arithmetic. Proposition~\ref{densite pc nulle alpha arithmetique} asserts that if some density $\alpha\in[0,1]$ of $\fkp$-curvatures vanish, then the formal subschemes are $(1-\alpha)$-arithmetic. Theorem~\ref{TH SL DEGRES} is then a link between the classic theorem of Schneider-Lang and an algebraicity criterion for a formal leaf of a foliation whose almost all $\fkp$-curvatures vanish, which is very close to a theorem of J.-B.~Bost in \citep{bost_algebraic-leaves}. It also generalizes an unpublished result of A.~Thuillier.

We also impose a condition of \emph{uniformization of order~$\rho$} of the formal subschemes at one place of the number field, which generalizes the hypothesis of having a curve parameterized by meromorphic functions of order~$\rho$ on the complex affine line. This condition is introduced in Paragraph~\ref{section ordre de croissance} (Definition~\ref{uniformisation}) and consists in the existence of a holomorphic map from an algebraic projective curve over~$\bC$ or~$\bC_p$ minus a finite subset of points $\tau\in T$ to~$X$ parameterizing the formal subschemes.
For such a map we also define a notion of \emph{order of growth near the singularities $\tau\in T$}, generalizing the notion of order of growth at infinity of a meromorphic function on~$\bC$ (Paragraph~\ref{ordre de croissance}, Definition~\ref{def ordre de croissance}).

Let us now state the main result (Theorem~\ref{TH SL DEGRES}) of this article.

\begin{theo*}
Let $X$ be a projective variety defined over~$\bQ$ and let $x_1,\dots,x_m$ be closed points of~$X$. For all $j\in\{1,\dots,m\}$, denote~$K_j=\bQ(x_j)$ the residue field of~$x_j$ and~$d_j$ its degree over~$\bQ$. Let $\alpha_1,\dots,\alpha_m$ be non-negative real numbers. For every $j\in\{1,\dots,m\}$, let~$\widehat V_j$ be a smooth $\alpha_j$-arithmetic $K_j$-subscheme of dimension~$1$ of the formal completion~$\widehat X_{x_j}$ of~$X$ at~$x_j$. Assume that the family of formal subschemes $(\widehat V_1,\dots,\widehat V_m)$ admits a uniformization of order at most~$\rho\geqslant 0$ at some finite or Archimedean place~$p_0$ of~$\bQ$.
Let $r$ be the dimension of the Zariski closure of 
~$\widehat V=\bigcup_{j=1}^m \widehat V_j$ in~$X$.

Then, 
\begin{itemize}
\item either $r>1$ and \[\sum_{j=1}^m\frac{1}{\alpha_jd_j}\leqslant \frac{r}{r-1}\rho,\]
\item or $r=1$, that is to say $\widehat V_1,\dots, \widehat V_m$ are all algebraic. 
\end{itemize}
\end{theo*}

If the variety~$X$ is $\bP^n(\bC)$ and the formal subschemes $\widehat V_j$ are parameterized by meromorphic functions on~$\bC$ of order at most~$\rho$, we recover Theorem~1 in~\citep{bertrand_SLdegres}, which is a generalization of the Schneider-Lang theorem dealing with the set of all points simultaneously mapped to algebraic points by the meromorphic functions. The set of such points is not always finite, but the theorem gives an inequality involving the degrees of the points.
Theorem~\ref{TH SL DEGRES} also generalizes an other result of D.~Bertrand, who proved in the article~\citep{bertrand_SLsurdomainesnonsimplementconnexes} a Schneider-Lang theorem on the projective space~$\bP^1$ minus a finite subset of points, in both the complex and $p$-adic cases. It also generalizes~\citep{diamond_SL} and~\citep{wakabayashi_sl} in which I.~Wakabayashi treats the case where the curve is uniformized by the complement of a finite set of points in a compact complex Riemann surface.

In the complex case, that is when the particular place~$v_0$ is the Archimedean place, Theorem~\ref{TH SL DEGRES} is a particular case of a theorem of C.~Gasbarri (Theorem~4.16 and Corollary~4.17 in~\citep{gasbarri_SL}), which holds for holomorphic functions on a \emph{parabolic curve}, where parabolic is intended in the sense of Ahlfors classification for Riemann surfaces (\citep{ahlfors_riemannsurfaces}), every algebraic curve being parabolic. The order of growth of such a function is then defined by the Nevanlinna theory. It would be an interesting problem to extend Theorem~\ref{TH SL DEGRES} into a non-Archimedean analog of C.~Gabsarri's result, which would require the definition and the study of parabolic analytic $p$-adic curves, for example in the framework of Berkovich spaces.

\bigskip

We now go in the substance of this article.
The proof of Theorem~\ref{TH SL DEGRES} makes use of the \emph{slopes method} invented by J.-B.~Bost, which requires the language of Arakelov geometry. For more details about this method and some elements of Arakelov geometry we refer the reader to~\citep{bost_bourbaki96, chambert-loir_algebricite, bost_algebraic-leaves, chen_these, bost_slopes, viada_these, viada_slopes}. We fix an ample line bundle~$L$ on~$X$ and the \emph{evaluation morphism} $\varphi^k_{D,\widehat V_j}$ maps a section of~$L^{\otimes D}$ which vanishes with order~$n_k$ along~$\widehat V_j$ to the~$n_{k+1}$-th ``Taylor coefficient'' of its restriction to $\widehat V_j$.

The sections of~$L^D$ are filtered by their order of vanishing along the formal subschemes $\widehat V_1,\dots,\widehat V_m$. At one step of the filtration, we do not require the same order of vanishing along the different formal subschemes. The ``derivation speed'' along one formal subscheme will be inversely proportional to the degree of the corresponding algebraic point. As far as we know, such a filtration with different speeds had not been used before, and it could hopefully have other applications to others settings.

A \emph{slopes inequality} reflects the fact that the formal subschemes are Zariski dense in~$X$. It involves the heights of the evaluation morphisms, and the conclusion of the proof of the main theorem follows from this very general slopes inequality combined with upper bounds of the heights coming from the two hypothesis made on the formal subschemes.


Our text is organized as follows. In the first section, we define the notion of \emph{$\alpha$-analytic formal subscheme}. This notion concerns the size of the $p$-adic absolute values of series parameterizing the formal subscheme. It is the notion of {$LG$-germ of type $\alpha$} in~\citep{gasbarri_SL}, and it is stronger than the notion of \emph{$\alpha$-arithmetic formal subscheme}. In some sense, this condition follows the idea of Schneider in his initial statement, in which there was no condition of global differential equation but arithmetic conditions on the Taylor coefficients of the functions at the particular points. In this paragraph we give some details and properties of this notion.

Then, in Section~\ref{alpha arithmetic} we define the \emph{evaluation morphisms} along a formal subscheme and the notion of \emph{$\alpha$-arithmetic formal subscheme} (Definition~\ref{def alpha-arithmetique}), which is a condition on the heights of the evaluation morphisms. The definitions and properties we establish apply to formal subschemes of any dimension, even if we will only need the case of formal subschemes of dimension~1 in view of Theorem~\ref{TH SL DEGRES}, because there is no additional difficulty and we intend to use formal subschemes of higher dimensions in prospective works (including \citep{herblot_SLproduit}, in preparation). We show that the $\alpha$-analyticity implies the $\alpha$-arithmeticity, and give a counter-example to the converse.

Section~\ref{section feuilletage} is devoted to the case of formal subschemes which are the the germs of formal leaves of an algebraic foliation at closed points. In particular, this makes the link between the classical statement and ours. In this section, we will see that such a formal subscheme is usually 1-analytic but not better, whereas it can be $\alpha$-analytic with $\alpha$ smaller than~1 under assumptions on the density of vanishing $\fkp$-curvature of the foliation. In Section~\ref{section points fermes} we develop these notions of $\alpha$-arithmeticity and $\alpha$-analyticity for formal subschemes based at any closed point, non-necessarily rational, and define evaluation morphisms in that case, which had not been done before.

Paragraphs~\ref{section ordre de croissance} and \ref{enonce SL degres} are devoted to the definitions of uniformization and of order of growth of a holomorphic map on an affine curve.

The proof of Theorem~\ref{TH SL DEGRES} consists in showing that the heights of the evaluation morphisms associated to the formal subschemes  satisfy some upper bounds (Proposition~\ref{maj hauteur}). More precisely, the needed upper bound is the combination of two different upper bounds with different origins: one comes from the condition that the formal subschemes are $\alpha$-arithmetic (Lemma~\ref{alpha arithmetique degre}), and the other uses the \emph{uniformization of order~$\rho$} of the family of formal subschemes at one place (Lemma~\ref{bonne maj morphisme d'evaluation}). 
Like in the classic theorem, it is at this point of the proof that the analytic estimates, as a \emph{Schwarz lemma}, appear. In the classic theorem, the estimation comes from the maximum principle applied on a ``big'' disk. This ``big'' disk can also be seen as the complement of a small disk containing the point at infinity. This is the idea we will use here: we take off well-chosen small ``disks'' containing the points~$\tau\in T$.

\bigskip
We introduce now some notation we will use in this text.

 Let $K$ be a number field and~$\fko_k$ its ring of integers. Let~$\Sigma_K$ be the set of places of~$K$. They are of two types: the \emph{finite places} corresponding to the maximal ideals of~$\fko_K$ and the \emph{Archimedean places} corresponding to the $[K:\bQ]$ embeddings of~$K$ in~$\bC$. With each maximal ideal~$\fkp$ of~$\fko_K$ we associate a  $\fkp$-adic absolute value~$| \cdot |_{\fkp}$ on~$K$ normalized in the following way: let~$\upvarpi$ be a uniformizing element, then \[|\upvarpi|_{\fkp}=N(\fkp)^{-1},\] where~$N(\fkp)$ is the norm of the ideal~$\fkp$, that is to say the cardinality of~$\fko_k/\fkp.$ If~$K_{\fkp}$ denotes the completion of~$K$ for this absolute value and~$p$ the prime number such that $(p)=\fkp\cap\bZ,$ we have
	\begin{equation}\label{normalization norme p-adique}|p|_{\fkp}=p^{-[K_{\fkp}:\bQ_p]}.
	\end{equation}
Every embedding $\sigma:K\hookrightarrow\bC$ defines an absolute value on~$K$ by $|x|_{\sigma}:=|\sigma(x)|,$ where $| \cdot |$ is the usual absolute value on~$\bC.$
 
Let $x\in K\setminus\{0\}$. With this normalizations, the \emph{product formula} is:
 \begin{equation}\label{formule du produit}
 \prod_{\fkp\in\Specmax\fko_K}|x|_{\fkp} \prod_{\sigma:K\hookrightarrow\bC}|x|_{\sigma}=1.
 \end{equation}

\section{ $\alpha$-analytic formal subschemes}\label{alpha analytique}

This notion of \emph{$\alpha$-analytic formal subscheme} is due to C.~Gasbarri; it is called \emph{$LG$-germ of type $\alpha$} in his article~\citep{gasbarri_SL}. We give here some details about this condition.

Let~$K$ be a number field. If~$A$ is a commutative unit ring, a n-tuple of formal series $f=(f_1,\dots,f_n)$ in~$n$ variables with coefficients in~$A$ is invertible for the composition law if and only if $f(0,\dots,0)=(0,\dots,0)$ and $Df(0)\in\GL_n(A)$. Hence the group of automorphisms $\Aut(\hat{\bA}^n_{K,0})$ of the formal completion of~$\bA^n_{K}$ at~0 can be identified to the $n$-tuples of formal series in~$n$ variables $f=(f_1,\dots,f_n)$, $f_i\in K[[X_1,\dots,X_n]]$ such that $f(0)=0$ and 
\[D f(0)=\left[\frac{\partial f_i}{\partial x_j}(0)\right]_{1\leqslant i,j\leqslant n}\in\GL_n(K).\]

\begin{definition} 
Let~$G_{\rm{an}}$ denote the subgroup of $\Aut(\hat{\bA}^n_{K,0})$ of the formal automorphisms $f=(f_1,\dots,f_n)\in\Aut(\hat{\bA}^n_{K,0})$ such that, for all $i\in\{1,\dots,n\}$, the series~$f_i$ has a positive radius of convergence at each place of~$K$. \end{definition}

 If $I=(i_1,\dots,i_n)$ is a multi-index, we define the factorial of~$I$ as $I!=\prod_{j=1}^ni_j!$.
 
\begin{definition} 
For any $a\geqslant 0$, let $G_{{\rm an},a}$ be the subset 
of formal automorphisms $f=(f_1,\dots,f_n)\in G_{\rm{an}},$ $f_i=\sum_I f_{i,I}X^I,$ such that there exist a finite subset~$S$ of place of~$K$, containing all Archimedean places, and a family $(C_{\fkp})_{\fkp\notin S}$ of real numbers such that
	\[\forall\fkp\notin S, C_{\fkp}\geqslant 1 \text{ and } \prod_{\fkp\notin S}C_{\fkp}<\infty,\]
and, for all  $\fkp\notin S$,
	\[\|f_{i,I}\|_{\fkp}\leqslant \frac{C_{\fkp}^{|I|}}{\|I!\|_{\fkp}^a}.\]
\end{definition}

\begin{lemme}\label{polygone de Newton}
Let $a\in\bR_+$ and let $f=(f_1,\dots,f_n)\in G_{{\rm an},a}$. Then there exists a finite subset~$S$ of finite places of~$K$ and a family $(C_{\fkp})_{\fkp\notin S}$ of real number at least equal to~1 such that for every $j\in\{1,\dots,m\}$ the radius of convergence of $f_j$ is at least ${C_{\fkp}}^{-1}p^{-\frac{a[K_{\fkp}:\bQ_p]}{p-1}}$ and
\[\text{for all } r\in\mathopen[0,{C_{\fkp}}^{-1}p^{-\frac{a[K_{\fkp}:\bQ_p]}{p-1}}\mathclose[, \sup_{|z|_p\leqslant r}|f_i(z)|_{\fkp}\leqslant C_{\fkp}r,\]
and \[\prod_{\fkp\notin S}C_{\fkp}<\infty.\]
\end{lemme}
\begin{proof} Let $j\in\{1,\dots,n\}$ and let $S$ and $(C_{\fkp})_{\fkp\notin S}$ be as in the definition. Then $f_j$ has a radius of convergence at least equal to ${C_{\fkp}}^{-1}p^{-\frac{[K_{\fkp}:\bQ_p]a}{p-1}}$, and for all real numbers $r<{C_{\fkp}}^{-1}p^{-\frac{a[K_{\fkp}:\bQ_p]}{p-1}}$ we have
\begin{align*}
\sup_{|z|_p\leqslant r}|f_i(z)|_{\fkp}&\leqslant \max_I |f_{i,I}|_{\fkp}r^{|I|}\leqslant rC_{\fkp}\max_I(C_{\fkp} r p^{\frac {a[K_{\fkp}:\bQ_p]}{p-1}})^{|I|-1}\leqslant C_{\fkp}r.\qedhere
\end{align*}
\end{proof}

Let $X$ be a projective algebraic variety of dimension~$n$ defined over a number field~$K$ and let~$P$ be a smooth $K$-rational point of~$X$. Let $\widehat V$ be a smooth formal subscheme of dimension~$d$ of the formal completion~$\widehat X_P$ of~$X$ at~$P$.

\begin{theo}There is a unique way of associating with such a triple $(X,\widehat V, P)$ a number $\alpha(X,\widehat V, P)$ in $\bR_+\cup\{\infty\}$ such that:
\begin{enumerate}
\item If $(X,\widehat V,P)=(\bA_K^n,\widehat V,(0,\dots,0))$, $\alpha(X,\widehat V, P)$ is the infimum in~$\bR_+\cup\{+\infty\}$ of the $a\in\bR_+$ such that there exists $f\in G_{{\rm an},a}$ such that $f^*{\widehat V}=\bA^d.$ (If the set of such $a$'s is empty, $\alpha(X,\widehat V, P)=\infty$.)
\item If $X\rightarrow X'$ is a closed immersion, then $\alpha(X,\widehat V, P)=\alpha(X',\widehat V, P).$
\item If there exist a triple $(X',\widehat V',P')$ and a morphism $X\rightarrow X'$, étale at $P$, mapping $P$ on $P'$ and inducing an automorphism $\widehat V\simeq\widehat V'$, then $\alpha(X,\widehat V, P)=\alpha(X',\widehat V', P').$

\end{enumerate}
\end{theo}

\begin{proof} We first show that Conditions 1, 2 and 3 determine $\alpha(X,\widehat V,P)$ for all $(X,\widehat V,P)$. Let $(X,\widehat V, P)$ be a triple, $U$ an open subset of~$X$ containing~$P$, and $f:U\to \bA_K^n$ an étale morphism mapping~$P$ to~$0$. Then, if follows from point 3 that $\alpha(U,\widehat V, P)=\alpha(\bA^n_K,f_*\widehat V,0)$ which is well-defined because of~1. Since the inclusion $(U,P)\hookrightarrow (X,P)$ is étale, we get also $\alpha(X,\widehat V, P)=\alpha(U,\widehat V, P)$ from point~3.

Let $f:U\rightarrow\bA^n_K$ be étale at~$P$. We prove that $\alpha(\bA^n_K,f_*\widehat V,0)$ does not depend on the choice of such an~$f$. There is a model $\mathscr U$  of $U$, of finite type over $\Spec\mathfrak o_K[\frac{1}{N}]$, and an étale morphism $\phi\colon\mathscr U\rightarrow\bA^n_{\mathfrak o_K[\frac{1}{N}]}$ whose restriction to the generic fiber is equal to $f$ and such that the rational point~$P$ extends to a section $\mathscr P$ of the morphism $\mathscr U\rightarrow\Spec\fko_K[\frac{1}{N}]$.
\begin{center}
\begin{tikzpicture}[decription/.style={fill=white, inner sep=2pt}]
\matrix(m)[matrix of math nodes, row sep=3em, column sep=2.5em, text height=1.5ex, text depth=0.25ex]
{\mathscr U&U\\
\Spec\fko_K[\frac{1}{N}]&\Spec K\\};
\path[->,font=\scriptsize]
(m-1-1) 	edge (m-2-1)
		edge (m-1-2)
(m-2-1)	edge (m-2-2)
		edge [bend left=30] node[auto] {$\mathscr P$} (m-1-1)
(m-1-2)	edge (m-2-2);		
\end{tikzpicture}
\end{center}

This étale morphism induces an \emph{isomorphism} of formal subschemes $\hat \phi:\hat{\mathscr U_{\mathscr P}}\rightarrow\hat{\bA}^n_{\mathfrak o_K[\frac{1}{N}],0}$ (see for instance \citep{liu_algebraicgeometry}, 4.3.2. Prop 3.26).

If $f$ and $g$ are two morphisms from~$U$ to $\bA^n_K$ étale at~$P$, we can choose~$\phi$ and~$\gamma$ as above, defined over the same model~$\mathscr U$.
Then, formally inverting~$\phi$,  we have $\hat \gamma\circ\hat \phi^{-1}\in\mathfrak o_K[\frac{1}{N}][[X_1,\dots,X_n]]^n$.
 
 \begin{center}
\begin{tikzpicture}[decription/.style={fill=white, inner sep=2pt}]
\matrix(m)[matrix of math nodes, row sep=3em, column sep=2.5em, text height=1.5ex, text depth=0.25ex]
{U&\bA^n_K\\
\bA^n_K&\\};
\path[->,font=\scriptsize]
(m-1-1) 	edge node[auto]{$f$}(m-1-2)
		edge node[auto] {$g$}(m-2-1);		
\end{tikzpicture}
\hspace{3em}
\begin{tikzpicture}[decription/.style={fill=white, inner sep=2pt}]
\matrix(m)[matrix of math nodes, row sep=3em, column sep=2.5em, text height=1.5ex, text depth=0.25ex]
{\hat{\mathscr U}_P&\hat{\bA}^n_{\mathfrak o_K[\frac{1}{N}],0}\\
\hat{\bA}^n_{\mathfrak o_K[\frac{1}{N}],0}&\\};
\path[->,font=\scriptsize]
(m-1-1) 	edge node[auto]{$\hat \phi$}(m-1-2)
		edge node[auto] {$\hat \gamma$}(m-2-1)
(m-1-2)	edge node[auto]{$\hat \gamma\circ\hat \phi^{-1}$} (m-2-1);		
\end{tikzpicture}
\end{center}
 
For every prime ideal $\fkp$ of~$\mathfrak o_K$ containing no prime factor of~$N$, the $\fkp$-adic norm of the coefficients of~$\hat g\circ \hat f^{-1}$ is at most~1, and hence \[\alpha(\bA^n_K,f_*\widehat V,0)=\alpha(\bA^n_K,g_*\widehat V,0).\qedhere\]
\end{proof}

\begin{definition}
Let $(X,\widehat V,P)$ be a triple as above, and let~$\alpha$ be a non-negative real number.  We will say that the formal subscheme~$\widehat V$ is \emph{$\alpha$-analytic} if $\alpha\geqslant\alpha(X,\widehat V,P).$
\end{definition}

\section{ $\alpha$-arithmetic formal subschemes}\label{alpha arithmetic}

\subsection{Evaluation morphisms} \label{section morphisme evaluation}

Let $X$ be a projective variety over~$K$. Let $P\in X(K)$ and $\widehat V$ be a smooth formal subscheme (of dimension $d$) of the formal completion~$\widehat X_{P}$ of~$X$ at~$P$. For all non-negative integers~$k$, let $(V)_k$ denote the $k$-th infinitesimal neighborhood of~$P$ in~$\widehat V$. Hence we have

\[\{P\}=(V)_0,\] \[(V)_k\subseteq (V)_{k+1},\]
\[\widehat V=\lim\limits_{\substack{\rightarrow\\k}}(V)_k.\]

Let $L$ be an ample line bundle on~$X$. For all non-negative integers $D,k$ we define the following $K$-vector spaces and $K$-linear applications:
\[E_D=\Gamma(X,L^{\otimes D}),\]
\begin{align*}
\eta_D:E_D&\rightarrow \Gamma(\widehat V,L^D)\\
s&\mapsto s_{|\widehat V},
\end{align*}
\begin{align*}
\eta_D^k:E_D&\rightarrow \Gamma((V)_k,L^D)\\
s&\mapsto s_{|(V)_k}.
\end{align*}

The vector spaces \[E^k_D=\ker\eta_D^{k-1}=\{s\in\Gamma(X,L^{\otimes D})\ |s_{|(V)_{k-1}}=0\}\] define a descending filtration of~$E_D$.

The kernel of the restriction map from $\Gamma((V)_k,L^D)$ to $\Gamma((V)_{k-1},L^D)$ is isomorphic to $\Sym^k\left(\Omega^1_{\widehat{V}}\right)\otimes L^D_{P}$ (see~\citep{viada_these}, Paragraph~4.2.5 or~\citep{viada_slopes}, Paragraph~2.2). Thus the map $\eta^k_D$ restricted to $E^k_D$ induces  linear map
	\begin{equation}\label{morphisme evaluation 1 point} 
	\varphi^k_{D,\widehat V}:E^k_D\longrightarrow\Sym^k\left(\Omega^1_{\widehat{V}}\right)\otimes L^D_{P},\end{equation}
which, roughly speaking, maps a section of $L^D$ vanishing at~$P$ with order~$k$ along~$\widehat V$ to the $(k+1)$-th ``Taylor coefficient'' of its restriction to~$\widehat V.$ By definition, the kernel of $\varphi_{D,\widehat V}^k$ is equal to~$E_D^{k+1}.$

\subsection{Integral structures, Hermitian structures, heights}\label{fibres vectoriels hermitiens}

Let $\mathscr X$ be a projective model of~$X$ over $\Spec(\mathfrak o_K)$, \ie a projective scheme over $\Spec(\mathfrak o_K)$ whose generic fiber $\mathscr X_K$ is isomorphic to~$X$. The rational point~$P$ ~extends to a section~$\mathscr P$ of the morphism $\pi:\mathscr X\rightarrow\Spec\mathfrak o_K.$ Let $\overline{\mathscr L}$ be a Hermitian line bundle on~$\mathscr X$ whose restriction $L=\mathscr L_K$ to~$X$ is ample.

Set $\mathscr E_D=\Gamma(\mathscr X,\mathscr L^{\otimes D})$. It is a projective $\mathfrak o_K$-module of finite type.
Let $\check\t_{P}\widehat V$ be the image of $\mathscr P^*\Omega^1_{\mathscr X/\mathfrak o_K}$ by the map
\[\mathscr P^*\Omega^1_{\mathscr X/\mathfrak o_K}\rightarrow\left(\mathscr P^*\Omega^1_{\mathscr X/\mathfrak o_K}\right)_K\simeq \Omega^1_{X/K,P} \rightarrow (T_{P}\widehat V)^\vee.\]
The restriction to~$K$ of the projective $\mathfrak o_K$-module of finite type $\check\t_{P}\widehat V$ is isomorphic to $T_{P}\widehat V ^\vee.$ Equipped with the dual metrics of the metrics $\| \cdot \|_{\sigma}$ on the $\bC$-vector spaces $T_{P}\widehat V\otimes_{\fko_K,\sigma}\bC$, $\check\t_{P}\widehat V$ is a Hermitian vector bundle $\overline{\check\t_{P}\widehat V}$ on $\Spec\mathfrak o_K.$ Its symmetric powers naturally inherit a structure of $\fko_K$-Hermitian vector bundle; for all non-negative integers~$k$, let $\|\cdot\|_{\sigma,\Sym,k}$ denote the norm on $\Sym^k (T_{P}\widehat V)^\vee$ associated with an embedding~$\sigma$ of~$K$ in~$\bC$.

We also define a structure of Hermitian vector bundle over $\Spec\fko_K$ on $\mathscr E_D:=\Gamma(\mathscr X,\mathscr L^{\otimes D})$. It is a projective $\mathfrak o_K$-module of finite type, and for all embedding $\sigma:K\hookrightarrow\bC$ we define a metric on $\mathscr E_{D,\sigma}=\mathscr E_D\otimes_{\mathfrak o_K,\sigma}\bC$ by
	\[\|s\|_{\sigma,\infty}:=\sup_{x\in\mathscr X_{K,\sigma}(\bC)}\|s(x)\|_{\sigma}.\]

Following H.~Chen \citep{chen_heights} and \'E.~Gaudron \citep{gaudron_fibresadeliques} (Paragraph 4.2), let us consider the \emph{John norm}, denoted by $\| \cdot \|_{\sigma,J}$, 
associated with the norm $\| \cdot \|_{\sigma,\infty}$. By definition, this norm is, among the Hermitian norms at least equal to $\| \cdot \|_{\sigma,\infty}$, the norm whose unit has a minimal volume.
These norms can be compared in the following way:
	\begin{equation}\label{comparaison norme John}\| \cdot \|_{\sigma,\infty}\leqslant \| \cdot \|_{\sigma,J} \leqslant \sqrt{\rk(E_D)}\|\cdot \|_{\sigma,\infty}.\end{equation}
Equipped with these norms $\| \cdot \|_{\sigma,J}$, $\mathscr E_D$ has a structure of Hermitian vector bundle~$\overline{\mathscr E_D}$ over~$\Spec o_K.$

\begin{definition}
Let $\ol E$, $\ol F$ be two Hermitian vector bundles on~$\Spec\fko_K$ and let~$\varphi$ be a non-zero $K$-linear map from $E_K=E\otimes_{\fko_K} K$ to $F_K=F\otimes_{\fko_K} K$. Let~$v$ be a place of~$K$. The \emph{height of~$\varphi$ at the place place~$v$} is the logarithm of the operator norm of~$\varphi$ extended to a linear map from~$E_v$ to~$F_v$, where $E_v$ and $F_v$ are the completions of~$E_K$ and~$F_K$~at the place~$v$ :
\[h_v(\varphi)=\log\|\varphi\|_v=\log\left(\sup_{e\in E_v\setminus\{0\}}\frac{\|\varphi(e)\|_v}{\|e\|_v}\right).\]

The \emph{height of~$\varphi$} is the sum of the heights of~$\varphi$ at every place of~$K$ :
	\begin{equation}\label{def hauteur}
	h(\varphi) = \sum_{\fkp}h_{\fkp}(\varphi)+\sum_{\sigma:K\hookrightarrow\bC}h_{\sigma}(\varphi).\end{equation}
\end{definition}

This definition of height is the usual definition in Arakelov theory. It will be useful to rewrite it in a slightly different way, so as to make the Archimedean and ultrametric places play more similar roles, as in~\citep{chambert-loir_equidistributionsytemesdynamiques}.

If~$p$ is a prime number, we denote by~$\bC_p$ the completion of an algebraic closure of the field of $p$-adics~$\bQ_p$. We still denote by~$|\cdot|_p$ the unique absolute value on~$\bQ_p$ which extends the $p$-adic absolute value on~$\bQ.$ Then, with every embedding~$\sigma_p$ of $K$ in~$\bC_p$ we can associate an absolute value on~$K$ by setting, for~$x\in K$ :
\[|x|_{\sigma_p}=|\sigma_p(x)|_{p}.\]
Denote by~$\bC_{\infty}$ the field of complex numbers~$\bC$ and by~$|\cdot|_{\infty}$ the usual absolute value on~$\bC_{\infty}=\bC$. Then, the Archimedean absolute values on~$K$ extending the usual absolute value on~$\bQ$ are the $x\mapsto|\sigma(x)|_{\infty},$ for $\sigma:K\hookrightarrow\bC$.

\begin{prop}\label{hauteur somme plongements}
Let $\ol E$ and $\ol F$ be two $\fko_K$-Hermitian vector bundles and let~$\varphi$ be a non-trivial $K$-linear map from $E_K=E\otimes_{\fko_K} K$ to $F_K=F\otimes_{\fko_K} K$. Then
\begin{equation}\label{hauteur somme sur plongements}
	h(\varphi)=\sum_{p\leqslant\infty}\sum_{\sigma:K\hookrightarrow\bC_p}h_{\sigma}(\varphi),\end{equation}
where in the first sum the index~$p$ describes the union of the set of prime numbers and the singleton~$\{\infty\}$.
\end{prop}
\begin{proof} It follows from the definition of the heights and the normalization~\eqref{normalization norme p-adique} we chose for the $\fkp$-adic norm on~$K$.
%
%
%
\end{proof}

\begin{lemme}\label{comportement hauteur extension des scalaires}
Let $K$ be a number field, $\ol E,\ol F$ two $\fko_K$-Hermitian vector bundles and set $E_K=E\otimes_{\fko_K} K$ and $F_K=F\otimes_{\fko_K} K$. Let $\varphi:E_K\to F_K$ be a non-zero $K$-linear map. Let~$K'$ be a finite extension of~$K$. Then, for every maximal ideal~$\fkp$ of~$\fko_K$, we have
	\[\frac{1}{[K':\bQ]}\sum_{\substack{\fkq\in\Specmax\fko_{K'},\\\fkq|\fkp}}h_{\fkq}(\varphi\otimes_K K')=\frac{1}{[K:\bQ]}h_{\fkp}(\varphi).\]
\end{lemme}

\begin{proof}
Let~$K_{\fkp}$ be the completion of~$K$ for the $\fkp$-adic absolute value. The map~$\varphi$ extends to a map $K_{\fkp}$-linear we still denote by~$\varphi$. Since~$F_{\fkp}$ is a $K_{\fkp}$-vector space of finite dimension, there exists a positive integer~$n$ such that $n\varphi$ maps~$E_{\fko_{\fkp}}$ in~$F_{\fko_{\fkp}}$. There exist positive integers~$\ell$ and~$m$, there exist a basis $(e_1,\dots,e_{\ell})$ of~$F_{\fko_{\fkp}}$ and integers $a_1,\dots,a_m$ such that $(a_1e_1,\dots,a_me_m)$ is a basis of $\im(n\varphi_{|E_{\fko_{\fkp}}})$.

Then 
	\[\|n\varphi\|_{\fkp}=\max_{1\leqslant i\leqslant m}|a_i|_{\fkp}.\]
Let $\fkq$ be a prime ideal of~$\fko_{K'}$ lying above~$\fkp$. Then 
	\begin{equation}\label{h(phi) et extension des scalaires}
	\|n\varphi\otimes_K K'\|_{\fkq}=\max_{1\leqslant i\leqslant m}|a_i|_{\fkq}=\max_{1\leqslant i\leqslant m}{|a_i|_{\fkp}}^{e_{\fkq}f_{\fkq}}={\|n\varphi\|_{\fkp}}^{e_{\fkq}f_{\fkq}},
	\end{equation}
where $f_{\fkq}$ is the residue class degree of~$\fkq$ over~$\fkp$, $e_{\fkq}$ the ramification index and their product~$e_{\fkq}f_{\fkq}$ is equal to the local degree~$[K'_{\fkq}:K_{\fkp}]$. Hence,
	\begin{align*}
	\sum_{\substack{\fkq\in\Specmax\fko_{K'},\\\fkq|\fkp}}\log\|\varphi\otimes_K K'\|_{\fkq}&=\sum_{\substack{\fkq\in\Specmax\fko_{K'},\\\fkq|\fkp}}[K'_{\fkq}:K_{\fkp}]\log \|\varphi\|_{\fkp}\\
	&=\log\|\varphi\|_{\fkp}[K':K].\qedhere
	\end{align*}
\end{proof}

Now we come back to the evaluation morphisms $\varphi^k_{D,\widehat V}$ defined by Formula~\eqref{morphisme evaluation 1 point}. Denote by $h_J(\varphi)$ the height of~$\varphi^k_{D,\widehat V}$ with respect to the John norms, that is to say $h_J(\varphi^k_{D,\widehat V})=\sum_{\fkp}\log\|\varphi^k_{D,\widehat V}\|_{\fkp}+\sum_{\sigma:K\hookrightarrow\bC}\log\|\varphi^k_{D,\widehat V} \|_{\sigma,J}$, where
	\[\|\varphi^k_{D,\widehat V} \|_{\sigma,J}=\sup_{s\in E_{\sigma}\setminus\{0\}}\frac{\| \varphi^k_{D,\widehat V}(s)\|_{\sigma.}}{\|s\|_{\sigma,J}}.\]

We also define the height $h(\varphi^k_{D,\widehat V})$ obtained replacing the Hermitian norms on $\mathscr E_{D,\sigma}$ by the infinity norm, and keeping the same norms $\|\cdot\|_{\sigma,\Sym^k}$ on $\Sym^k\left(\Omega^1_{\widehat{V}}\right)\otimes L^D_{P}$,
\[h(\varphi^k_{D,\widehat V})=\sum_{\fkp}\log\|\varphi^k_{D,\widehat V}\|_{\fkp}+\sum_{\sigma:K\hookrightarrow\bC}\log\|\varphi^k_{D,\widehat V}\|_{\sigma,\infty},\]
where $\|\varphi^k_{D,\widehat V}\|_{\sigma,\infty}=\sup_{s\in E^k_{D,\sigma}\setminus\{0\}}\frac{\left\|\varphi^k_{D,\widehat V}(s)\right\|_{\sigma,\Sym^k}}{\|s\|_{\sigma,\infty}}.$

From~\eqref{comparaison norme John},
we have \begin{equation}\label{majoration hauteur john par norme infinie}h_J(\varphi^k_{D,\widehat V})\leqslant h(\varphi^k_{D,\widehat V}).\end{equation}

\subsection{$\alpha$-arithmetic formal subschemes} \label{section alpha arithmetique}

Let~$X$ be a projective variety defined over a number field~$K$, let~$P$ be a $K$-rational point of~$X$ and let~$\widehat V$ be a smooth formal subscheme of~$\widehat X_P$. Let~$L$ be an ample line bundle on~$X$. For all couples of non-negative integers $(k,D)$, denote by~$\varphi^k_{D,\widehat V}$ the associated evaluation morphism:
\[\varphi^k_{D,\widehat V}:E^k_{D,\widehat V}\to \Sym^k\left(\Omega^1_{\widehat V}\right)\otimes L^D_{|P},\]
where $E^k_{D,\widehat V}=\{s\in H^0(X,L^D)\ |\ s_{|(V)_{k-1}}=0\}.$

\begin{definition}\label{def S,alpha-arithmetique}
Let $\alpha$ be a non-negative real number and $S$ be a finite subset of (finite or Archimedean) places of~$K$. A smooth formal subscheme~$\widehat V$ is said to be \emph{$(S,\alpha)$-arithmetic} if for all~$\ol\alpha>\alpha$, there exist $C>0$ and a family of non-negative real numbers $(C_v)_{v\in\Sigma_K}$ such that, for all non-negative integers~$D$ et~$k$, the evaluation morphism $\varphi_{D,\widehat V}^k$ satisfies
	\begin{equation}\frac{1}{[K:\mathbf Q]}\sum_{v\in\Sigma_K\setminus S}h_{v}(\varphi_{D,\widehat V}^k)\leqslant \ol\alpha k\log k +C(k+D)
	\label{S,alpha arithmetique},\end{equation}
and for each place~$v$ of~$K$,
	\[h_{v}(\varphi_{D,\widehat V}^k)\leqslant C_v(k+D).\]
\end{definition}

\begin{definition}\label{def alpha-arithmetique}
Let $\alpha$ be a non-negative real number. A smooth formal subscheme~$\widehat V$ is \emph{$\alpha$-arithmetic} if it is $(S,\alpha)$-arithmetic for all finite subsets~$S$ of places of $K$.
\end{definition}

\begin{lemme}\label{alpha-arithmetique independant du corps}
Let~$X$ be a projective variety over a number field~$K$, let~$P$ be a $K$-rational point of~$X$ and~$\widehat V$ be a smooth formal subscheme of~$\widehat X_P$. Let~$K'$ be a finite extension of the field~$K$ and let~$\alpha$ be a non-negative real number.

Then the formal subscheme~$\widehat V$ is $\alpha$-arithmetic if and only if $\widehat V\otimes_K K'$ is $\alpha$-arithmetic.
\end{lemme}

\begin{proof} The proof follows from the definition of $\alpha$-arithmeticity and the following fact. Write $\pi:\Spec\fko_{K'}\to\Spec\fko_K$, let $\fkq$ be a maximal ideal of~$\fko_{K'}$ and $\fkp=\pi(\fkq)$ the maximal ideal of~$\fko_K$ lying under~$\fkq$, then, from~\eqref{h(phi) et extension des scalaires} one has
	\begin{equation*}
	h_{\fkq}(\varphi^k_{D,\widehat V}\otimes_K K')=[K'_{\fkq}:K_{\fkp}]h_{\fkp}(\varphi^k_{D,\widehat V}).\qedhere
	\end{equation*}
\end{proof}

The evaluation morphism defined in Paragraph~\ref{section morphisme evaluation} and denoted by $\varphi_{D,\widehat V}^k$ depends on the choice of the ample line bundle~$L$ on~$X$ and on the choice of integral models of~$X$ and~$L$ over $\Spec\fko_K$.

The following propositions imply the fact that, for a formal subscheme,
satisfying Definition~\ref{def alpha-arithmetique} 
does not depend on these choices.
The arguments come from the article~\citep{bost_chambert-loir_analyticcurves}, Proposition~4.7.

The independence of the choice of models can be proved exactly in the same way as the part a)
of this proposition; changing the model only modify the left side by a term bounded from above by $C(k+D)$.
Let us handle with the independency on the line bundle. We precise with an index the line bundle with respect to which the evaluation morphism is defined: thus, for all non-negative integers $k,D$, we write $E_{D,L}=\Gamma(X,L^{\otimes D})$ and $\varphi^k_{D,\widehat V,L}$ the evaluation morphism $E^k_D\rightarrow\Sym^k(\Omega^1_{\widehat V_P})\otimes L_{|P}^{\otimes D}$. 

\begin{prop}
\begin{enumerate}
\item Let $b$ be a positive integer and~$L$ be an ample line bundle on~$X$. If $\varphi^k_{D,\widehat V,L}$ satisfies Inequality~\eqref{S,alpha arithmetique}, then so does~$\varphi^k_{D,\widehat V,L^{\otimes b}}$.
\item Let~$L$ and~$M$ be two ample line bundles on~$X.$ Assume there exists $\sigma\in\Gamma(X,M\otimes L^{-1})$ which does not vanish at~$P.$
Then there exists $C>0$ such that $\|\varphi^k_{D,\widehat V,L}\|_{\fkp}\leqslant \|\varphi_{D,\widehat V,M}^k\|_{\fkp}C^D.$
\end{enumerate}
\end{prop}

\begin{proof}
The evaluation morphism defined with respect to the line bundle~$L^{\otimes b}$, 
\[\varphi^k_{D,\widehat V,L^{\otimes b}}:E^k_{D,L^b}\rightarrow\Sym^k(\Omega^1_{\widehat V_P})\otimes L_{|P}^{\otimes Db},\]
coincides with $\varphi^k_{bD}$. Therefore their norms are equal, 
 and this proves the first point.

Let $s\in E^k_{D,L}$ be a section which vanishes with order~$k$ along~$\widehat V.$ Then $s\otimes\sigma^D\in E^k_{D,M}.$ Moreover,
\[\varphi^k_{D,M}(s\otimes\sigma^D)=\varphi^k_{D,L}(s)\otimes\sigma(P)^D.\]

Thus,
	\begin{align*}
	\|\varphi^k_{D,\widehat V,L}(s)\|&=\|\varphi^k_{D,\widehat V,M}(s\otimes\sigma^D)\|\cdot\|\sigma(P)\|^{-D}
	\leqslant \|\varphi^k_{D,\widehat V,M}\| \cdot\|s\|C^D,
	\end{align*} 
setting $C=\|\sigma\|\|\sigma(P)\|^{-1}.$\qedhere
\end{proof}

\begin{prop}\label{alpha analytique arithmetique}
Let $X$ be a projective variety over a number field~$K$, let $P$ be a smooth $K$-rational point of~$X$ and let $\widehat V$ be a smooth formal subscheme of~$\widehat X_P$. Let $\alpha\in\bR_+.$

If $\widehat V$ is $\alpha$-analytic, then $\widehat V$ is $\alpha$-arithmetic.
\end{prop}

\begin{proof}
Let $d$ denote the dimension of~$\widehat V$, and~$\Specmax(\fko_K)$ be the maximal spectrum of~$\fko_K$. Assume that $\widehat V$ is $\alpha$-analytic and let $\ol\alpha>\alpha$. Then~$\widehat V$ is parameterized by formal series $f_1,\dots,f_n\in K[[x_1,\dots,x_d]]$, $f_i=\sum_I a_I(i)x^I,$ which, at each place, have a positive radius of convergence and satisfy: there is a finite subset~$S$ of places of~$K$ ~containing all Archimedean places such that, for all $\fkp\in\Specmax(\fko_K)\setminus S$, there exists $C_{\fkp}>0$ such that, for all $I\in\bZ_{\geqslant 0}^d$, for all $i\in\{1,\dots,n\}$, 
	\begin{equation}\label{coef parametrage a-analytique} |a_I(i)|_{\fkp}\leqslant\frac{C_{\fkp}^{|I|}}{\|I!\|_{\fkp}^{\ol\alpha}},\end{equation}
	and \[\prod_{\fkp\in\Specmax(\fko_K)\setminus S} C_{\fkp}<+\infty.\]
For every place~$v$ of~$K$, there is a non-negative real number~$C'_v$ , $C'_v=1$ for almost every~$v$, such that
	\begin{equation}\label{maj hauteur en fonction coef parametrage}
	h_{v}(\varphi^k_{D,\widehat V})\leqslant C_v'(k+D)+\log\left(\max_{1\leqslant i\leqslant n}\max_{|I|=k}|a_I(i)|_{v}\right).
	\end{equation}
To give an upper bound for $\varphi^k_{D,\widehat V}$, we use Inequality~\eqref{coef parametrage a-analytique} which gives an upper bound for the height at every place~$v\in\Sigma_K\setminus S$, and the analyticity of the series $f_1,\dots,f_d$ gives an easy upper bound for~$h_v(\varphi^k_{D,\widehat V})$ at the places~$v$ in the finite set~$S$.

Thanks to the inequalities~\eqref{maj hauteur en fonction coef parametrage} et~\eqref{coef parametrage a-analytique},
\[ h_{\fkp}(\varphi_{D,\widehat V}^k)\leqslant C'_{\fkp}{(k+D)}+ \log \max_{I\in\bZ_{\geqslant 0}^d,\ |I|\leqslant k}\frac{C_{\fkp}^{|I|}}{\|I!\|_{\fkp}^{\ol\alpha}}.\]
If $I=(i_1,...,i_d)\in\bZ_{\geqslant 0}^d$, then
$|I|!/i_1!...i_d!$ is an integer, and hence $\frac{1}{\| I ! \|_{\fkp}} \leqslant \frac 1{|k!|_{\fkp}}$ if $|I|\leqslant k$.
Then,
\[ \log \|\varphi_{D,\widehat V}^k\|_{\fkp}\leqslant C'_{\fkp}{(k+D)}+ k  \log C_{\fkp} - \ol\alpha\log |k!|_{\fkp}. \]
If $C=\log\prod_{\fkp\notin S} C_{\fkp}$ et $C'=\sum C'_{\fkp}$, we get
	\begin{align}
	\sum_{\fkp\in\Sigma_K\setminus S} \log\| \varphi^k_{D,\widehat V}\|_{\fkp}& \leqslant C'(k+D)+C k -\ol\alpha \sum_{\fkp\in\Sigma_K\setminus S} \log |k!|_{\fkp}\nonumber \\
	\label{maj p pas dans S} &\leqslant C''(k+D)-\ol\alpha \sum_{\fkp\in\Sigma_K\setminus S} \log |k!|_{\fkp}.
	\end{align}

Let~$v$ be a place of~ $K$. Let $r_{v}(i)$ be a positive real number, less than the radius of convergence of~$f_i$. Then $|a_I(i)|_{v}r_{v}(i)^{|I|}\rightarrow 0$ when $|I|$ goes to infinity, and therefore there is a a positive real number~$C_v$ ~such that
\begin{equation}\label{maj triviale}h_{v}(\varphi^k_{D,\widehat V})\leqslant C_v(k+D).\end{equation}
Setting $C_0=\sum_{v\in S}C_v$, finite sum of positive terms, we get
\begin{equation}\label{maj v dans S}
\sum_{v\in S}h_{v}(\varphi^k_{D,\widehat V})\leqslant C_0(k+D).
\end{equation}

From the inequalities~\eqref{maj p pas dans S} and~\eqref{maj v dans S} for the heights,
	\begin{align*}
	h(\varphi^k_{D,\widehat V})&=\sum_{v\in S}h_{v}(\varphi^k_{D,\widehat V})+\sum_{\fkp\notin S} h_{\fkp}(\varphi^k_{D,\widehat V})
	\leqslant C_1(k+D)-\ol\alpha \sum_{\fkp\in\Specmax\fko_K} \log |k!|_{\fkp}.
	\end{align*}
The product formula implies \[-\sum_{\fkp\in\Specmax\fko_K} \log |k!|_{\fkp}=[K:\bQ]\log (k!)\leqslant [K:\bQ]k\log k,\] and hence
	\[h(\varphi^k_{D,\widehat V})=\sum_{v\in\Sigma_K}h_v(\varphi^k_{D,\widehat V})\leqslant C_1(k+D)+\ol\alpha[K:\bQ]k\log k.\]

Let $S'$ be a finite subset of places of~$K$. Since for every $v\in S'$, the height~$\varphi^k_{D,\widehat V}$ of the evaluation morphism at the place~$v$ satisfies the simple inequality~\eqref{maj triviale}, we also get, by the same arguments,
\[\sum_{v\in\Sigma_K\setminus S'}h_v(\varphi^k_{D,\widehat V})\leqslant C_2(k+D)+ \ol\alpha [K:\bQ] k\log k,\]
and this holds for any $\ol\alpha>\alpha$.
Therefore the formal subscheme~$\widehat V$ is \mbox{$\alpha$-arithmetic}.
\end{proof}

The converse of Proposition~\ref{alpha analytique arithmetique} is false, we will give a counterexample at the end of next paragraph, page~\pageref{contre-exemple}.

\section{Formal germs of leaves of an algebraic foliation}\label{section feuilletage}


Let $X$ be a projective variety over a number field~$K$ and let $d$ be a positive integer. A (regular) algebraic foliation of dimension~$d$ on an open subset~$U$ of~$X$ is a $d$-dimensional subbundle of the tangent bundle $TU$ which is involutive, that is to say stable under Lie brackets.

In this paragraph, we will study the case of formal subschemes~$\widehat V$ which are germs of formal leaves of an algebraic foliation on~$X$ through a rational point.

\begin{lemme}\label{equa diff implique 1-analytique}
Let $X$ be a projective variety over a number field~$K$ and let~$P$ be a smooth $K$-rational point of~ $X$. Let $F$ be an algebraic foliation on an open subset of~$X$ containing~$P$ and let $\widehat V$ be the germ of formal leaf defined by~$F$ at~$P$.

Then $\widehat V$ is 1-analytic.
\end{lemme}

\begin{proof}
Let us recall the definition and some properties of formal leaves of an algebraic foliation and of their parametrization, following~\citep{bost_algebraic-leaves}. There is an open subset~$U$ of~$X$ containing~$P$ and regular functions $x_1,\dots,x_n$ on~$U$ such that the map $(x_1,\dots,x_n):U\rightarrow \bA_K^n$ is étale and maps~$P$ to~$0$. We identify~$\widehat X_P$ with $\hat \bA^n_{K,0}=\Specf K[[x_1,\dots,x_n]]$ via the ``local coordinates''~$\hat x_j$. Let $(v^1,\dots,v^d)$ be a basis of~$F$ on an open neighborhood~$V$ of~$P$ of commuting vector fields $v^j\in\fko_K[[x_1,\dots,x_n]]^n$. In the complex analytic case, the proof of this result is a classic one, see for instance the appendix of~\citep{camacho_neto_foliations}, and is quite similar in the formal case (see~\citep{herblot_these}). For all $j\in\{1,\dots,d\}$, let $D_j$ be the derivation on $K[[x_1,\dots,x_n]]$ associated with $v^j$, and for every multi-index $I=(i_1,\dots i_d)\in\bZ_{\geqslant 0}^d$ let $D^I$ denote the differential operator $D^{i_1}_1\dots D^{i_d}_d$ and $I!=\prod_{j=1}^d i_j!$.

The formal leaf $\widehat V$ of~$F$ through~$P$ is parameterized by $\psi:\hat\bA^d_{K,0}\times\widehat X_P\longrightarrow\widehat X_P$ defined as
\begin{equation}
\psi(t_1,\dots,t_d,0,\dots,0)=\sum_{I\in\bZ_{\geqslant 0}^d}\frac{t^I}{I!}D^I(x_1,\dots,x_n)(P) \label{flot formel},
\end{equation}
where $t^I=\prod_{j=1}^d t_{j}^{i_j}$.

Let~$f$ be the morphism $f:\hat\bA^n_{K,0}\rightarrow\widehat X_P$ which is given by \begin{equation}\label{redressement V chapeau}
f(t_1,\dots,t_n)=\psi(t_1,\dots,t_d,0,\dots,0)+(0,\dots,0,t_{d+1},\dots,t_{n}),
\end{equation}
in terms of the local coordinates $x_1,\dots,x_n$ on $\widehat X_P$. lt satisfies $f^{-1}\widehat V=\hat \bA^d_{K,0}\times\{0\}^{n-d}.$ To show that the formal subscheme~$\widehat V$ is $1$-analytic, it is sufficient to prove that $f\in G_{{\rm an},1}$. To do so, it is sufficient to give an upper bound for the coefficients of the parametrization $\psi(t_1,\dots,t_d,0,\dots,0).$

For all $I\in\bZ_{\geqslant 0}^d$, 
\[D^I:\mathfrak o_K[[x_1,\dots,x_n]]\rightarrow\mathfrak o_K[[x_1,\dots,x_n]].\]

Hence we have 
	\begin{equation}\label{majoration sans condition pc}
	\left|\frac{1}{I!}D^I(x_1,\dots,x_n)(P)\right|_{\fkp}=|I!|_{\fkp}^{-1}|D^I(x_1,\dots,x_n)(P)|_{\fkp}\leqslant |I!|_{\fkp}^{-1},\end{equation}
which proves that $\widehat V$ is 1-analytic.
\end{proof}

\begin{remarque}
This inequality fulfilled by the coefficients of a parametrization of~$\widehat V$ is stronger than the condition of $1$-analyticity. Actually, for almost every place~$\fkp$ those coefficients are less or equal to~$\frac{C_{\fkp}^{|I|}}{|I!|_{\fkp}}$ with $C_{\fkp}=1$.
\end{remarque}

If $D$ is a derivation on a commutative ring~$A$ of positive characteristic~$p$, then from the Leibniz rule its $p$-th  composite $D^p$ is still a derivation on~$A$.
Let~$X$ be a projective variety over a number field~$K$. Let~$F$ be an algebraic foliation on a smooth open subset~$U$ of~$X$, defined over~$K$. Let $N$ be a positive integer such that there exist a smooth model~$\mathscr U$ of~
$U$ over~$\fko_K[1/N]$ and an involutive subbundle~$\mathscr F$ of~$T\mathscr U$
with generic fiber~$F$.
Let $\fkp$ be a maximal ideal of $\fko_K[1/N]$, denote by $\mathbf F_{\fkp}$
its residue field and by~$p$ the characteristic of $\mathbf F_{\fkp}$.
We say that $\mathscr F$ has \emph{vanishing $\fkp$-curvature}
if the subbundle $\mathscr F\otimes \mathbf F_{\fkp}$
of $T(\mathscr U\otimes\mathbf F_{\fkp})$ is stable under its $p$-th power.

For almost every~$\fkp$, this notion does not depend of the choices of~$\mathscr U$
and~$\mathscr F$ we made.
Hence we will say in that case that $F$ has \emph{vanishing $\fkp$-curvature}. (For more details about involutive vector bundles in positive characteristic, see~\citep{miyaoka_foliation}.)

\begin{lemme}
Let $X$ be a projective variety defined over a number field~$K$, and let $P$ be a smooth $K$-rational point of~$X$. Let $F$ be an algebraic foliation defined on an open subset of~$X$ containing~$P$ and let $\widehat V$ be the germ of formal leaf of~$F$ through~$P$. Assume that almost every $\fkp$-curvature (that is to say all but finitely many) of~$F$ vanishes.

Then $\widehat V$ is 0-analytic.
\end{lemme}

\begin{proof}
Let $\fkp$ be a maximal ideal of $\fko_K$ and $p$ be the prime number such that $(p)=\fkp\cap\bZ$. Let $f$ be the residue class degree, that is to say the degree of the field extension $\fko_K/\fkp$ over $\bF_p$, so that $\fko_K/\fkp=\bF_{p^f}.$ Let $\fko_{\fkp}$ be the completion of $\fko_K$ for the $\fkp$-adic absolute value. Let $\upvarpi$ be a  uniformizing element of~$\fko_{\fkp}$, $(p)=(\upvarpi)^e$ where $e$ is the absolute ramification index.

Let $I=(i_1,\dots,i_d)\in\bZ_{\geqslant 0}^d$. The derivation $D^I=D_1^{i_1}\dots D_d^{i_d}$ acts on $\fko_{\fkp}[[x_1,\dots,x_n]]^n$. Moreover, if the $\fkp$-curvature vanishes, $D_i^p$ maps $\fko_{\fkp}[[x_1,\dots,x_n]]^n$ into $\upvarpi\fko_{\fkp}[[x_1,\dots,x_n]]^n.$
Let $g\in\fko_{\fkp}[[x_1,\dots,x_n]]^n$. For all $j\in\{1,\dots,d\}$ we write $i_j=q_jp+r_j$ the Euclidean Division of $i_j$ by $p$ and set $q=\sum_{j=1}^d q_j$. Then

\begin{align*}
D^I(g)&=D_1^{q_1p+r_1}\dots D_d^{q_dp+r_d}(g)= {(D_1^{p})}^{q_1} D_1^{r_1}\dots {(D_d^{p})}^{q_d} (D_d^{r_d}(g)),
\end{align*}
and finally $D^I(g)\in\upvarpi^{q}\fko_{\fkp}[[x_1,\dots,x_n]]^n.$

The germ of formal leaf defined by the foliation at~$P$ is parameterized by
	\[\psi(t_1,\dots,t_d,x_1(P),\dots,x_n(P))=\sum_{I\in\bZ_{\geqslant 0}^d}\frac{t^I}{I!}D^IX(P),\]
with $X=(x_1,\dots,x_n)$.	

Now, we give an upper bound for the coefficients of the parametrization~$\psi$:

\begin{align}
\left|\frac{1}{I!}D^IX(0)\right|_{\fkp}&\leqslant |I!|^{-1}_{\fkp}|\upvarpi|_{\fkp}^{q}
\leqslant p^{[K_{\fkp}:\bQ_p]\left(\sum_{j=1}^d v_p(i_j!) -\frac{1}{e}\lfloor\frac{i_j}{p}\rfloor\right)}\label{coef parametrage}.
\end{align}

Recall the normalization of the $\fkp$-adic absolute value we chose. Let $\upvarpi$ be a  uniformizing element at $\fkp$, we have: 
\[|p|_{\fkp}=N(\fkp)^{-e}=p^{-[K_{\fkp}:\bQ_p]},\]
\[|\upvarpi|_{\fkp}=N(\fkp)^{-1}=p^{-f}=p^{-\frac{[K_{\fkp}:\bQ_p]}{e}}.\]

%
We recall that if $a$ is a non-negative integer, then $v_p(a!)-\left\lfloor\frac{a}{p}\right\rfloor\leqslant \frac{a}{p(p-1)}$, because the $p$-adic valuation of $a!$ is $v_p(a!)=\sum_{k=1}^{\infty} \left\lfloor\frac{a}{p^k}\right\rfloor$. If $e=1,$ this remark and the inequality~\eqref{coef parametrage} imply the following inequality:
\begin{align}
\left|\frac{1}{I!}D^IX(0)\right|_{\fkp}\leqslant p^{[K_{\fkp}:\bQ_p]\sum_{j=1}^d \frac{i_j}{p(p-1)}}
\leqslant p^{[K_{\fkp}:\bQ_p]\frac{|I|}{p(p-1)}}
=|p|_{\fkp}^{-\frac{|I|}{p(p-1)}}
\label{majoration coef parametrage pc=0}.
\end{align}

Let $S$ be the finite subset of the ramified ideals $\fkp$. For all $\fkp\notin S$, the ramification index~$e$ is equal to~1 and Inequality~\eqref{majoration coef parametrage pc=0} holds.
For $\fkp\notin S,$ we set 
	\[C_{\fkp}=p^{\frac{[K_{\fkp}:\bQ_p]}{p(p-1)}}.\]
The sum $\sum_{\fkp\notin S}\log C_{\fkp}=\sum\frac{\log p}{p(p-1)}[K_{\fkp}:\bQ_p]$ converges and 
\[\left|\frac{1}{I!}D^IX(0)\right|_{\fkp}\leqslant C_{\fkp}^{|I|},\]
so $\widehat V$ is $0$-analytic.
\end{proof}

\subsection{Density of vanishing $\fkp$-curvatures} \label{densite p-courbures nulles}

In this paragraph, we define a density $\beta$ in~$[0,1]$ of vanishing $\fkp$-curvatures of the foliation~$F$, which is equal to~1 without any additional assumption on the foliation and~0 if almost every $\fkp$-curvature vanishes. This density is related to the notion of $\alpha$-arithmetic formal subscheme: a formal leaf with a density~$\beta$  of vanishing $\fkp$-curvatures is $(1-\beta)$-arithmetic (Proposition~\ref{densite pc nulle alpha arithmetique}). With this definition, Theorem~\ref{TH SL DEGRES} gives, in the Archimedean case, a kind of ``interpolation'' between the classic Schneider-Lang theorem (when the density of vanishing $\fkp$-curvatures is zero) and (when almost every $\fkp$-curvature vanishes) an algebraicity theorem close to the theorem by \mbox{J.-B.}~Bost in his article~\citep{bost_algebraic-leaves}.

\begin{lemme} Assume that~$\widehat V$ is the germ of formal leaf of an algebraic foliation at a rational point. Then the evaluation morphism satisfies the following properties:
there exists a finite set $S$ of maximal ideals of~$\fko_K$ such that, for every $\fkp\in\Specmax \fko_K\setminus S$, for every $k\in\bZ_{\geqslant 0}$ and $D\in\bZ_{>0}$,
	\begin{equation}\label{morphisme d'evaluation si feuilletage}
	h_{\fkp}(\varphi_{D,\widehat V}^k)\leqslant k\frac{[K_{\fkp}:\bQ_p]\log p}{p-1},\end{equation}
and if moreover $k<p,$ we have:
	\begin{equation}\label{p<k}
	h_{\fkp}(\varphi_{D,\widehat V}^k)\leqslant 0.\end{equation}
Let $\fkp$ be a maximal ideal of~$\fko_K$, $\fkp\notin S$, such that the $\fkp$-curvature of~$F$ vanishes. Then
	\begin{equation}\label{morphisme evaluation pc=0}
	h_{\fkp}(\varphi_{D,\widehat V}^k)\leqslant k\frac{[K_{\fkp}:\bQ_p]\log p}{p(p-1)}.\end{equation}
\end{lemme}

\begin{proof}
The inequalities~\eqref{morphisme d'evaluation si feuilletage} and~\eqref{p<k} follow from Inequality~\eqref{majoration sans condition pc}. If moreover the $\fkp$-curvature vanishes, from Inequality~\eqref{majoration coef parametrage pc=0} we get~\eqref{morphisme evaluation pc=0}.
\end{proof}

For every~$x\in\bR_+^*$, we define 
	\[\psi_K(x)=\sum_{\substack{\fkp\in\Specmax\fko_K \\ \text{s.t. }  p\leqslant x}}[K_{\fkp}:\bQ_p]\frac{\log p}{p-1}.\]
	
\begin{lemme}\label{poids densite}
Let $K$ be a number field. Then, when $k$ goes to $+\infty$,
	\[\psi_K(k)\sim [K:\bQ]\log k.\]
\end{lemme}	

\begin{proof}
The function $\psi_K$ satisfies $\psi_K=[K:\bQ]\psi_{\bQ}$, and it is well-known that $\psi_{\bQ}(k)\sim\log k$ when $k$ goes to~$\infty$ (see for instance~\citep{tenenbaum} Chapter~I.1, Theorem~7).
\end{proof}	
	
	\begin{definition}\label{def densite}
Let $F$ be an algebraic foliation on~$X$. For every $x\in\bR_+$, we set
	\[\beta_x=\frac{1}{\psi_K(x)}\sum_{\substack{\fkp \text{ s.t. }  p\leqslant x,\\ \text{$\fkp$-curvature(F)}=0}}\frac{[K_{\fkp}:\bQ_p]\log p}{p-1}.\]
We call \emph{(inferior) density of vanishing $\fkp$-curvatures} the following real number between~0 et~1:
\begin{equation}\label{densite p-courbures}
\beta=\varliminf_{x\rightarrow\infty}\beta_x.
\end{equation}
\end{definition}

\begin{prop}\label{densite pc nulle alpha arithmetique}Let~$X$ be a projective variety defined over a number field~$K$ and let~$P$ be a smooth $K$-rational point of~$X$. Let $F$ be an algebraic foliation defined on an open subset of~$X$ containing~$P$ and let $\widehat V$ be the germ of formal leaf of~$F$ through~$P$.

Let $\beta$ be the density of vanishing $\fkp$-curvatures of~$F$. Then $\widehat V$ is $(1-\beta)$-arithmetic.
\end{prop}

\begin{proof}
Let~$S$ be a finite subset of places of~$K$. We want to show that for every $\varepsilon>0$, there exists a non-negative real number~$C$ such that, for all $k,D$,
	\[\sum_{v\in \Sigma_K\setminus S}h_{v}(\varphi^k_{D,\widehat V}) \leqslant C(k+D)+(1-\beta+\varepsilon)k\log k.\]
Since the formal subscheme~$\widehat V$ is the germ of leaf of an algebraic foliation, it is $1$-arithmetic and hence, for every place~$v$ of~$K$, there is a non-negative real number~$C_v$ such that, for every $k\in\bZ_{\geqslant 0}$ and $D\in\bZ_{>0}$,
	\[h_{v}(\varphi^k_{D,\widehat V}) \leqslant C_v(k+D).\]
Let $A$ be the set of maximal ideals $\fkp$ of $\fko_K$ such that the $\fkp$-curvature of~$F$ vanishes. Then, setting $C_1=\sum_{v:K\hookrightarrow\bC}C_v$,
	\begin{align*}
	\sum_{v\in \Sigma_K\setminus S}h_{v}(\varphi^k_{D,\widehat V}) 
	&\leqslant  C_1(k+D)+\sum_{\substack{\fkp\in\Specmax\fko_k\\ p\leqslant k}} h_{\fkp}(\varphi^k_{D,\widehat V})  \qquad \text{ from~\eqref{p<k}}\\
	&\leqslant C_1(k+D)+\sum_{\substack{\fkp\in A\\ p\leqslant k}} h_{\fkp}(\varphi^k_{D,\widehat V})
	+\sum_{\substack{\fkp\in\Specmax\fko_k\setminus A\\ p\leqslant k}} h_{\fkp}(\varphi^k_{D,\widehat V})\\
	&\leqslant C_2(k+D)+k\sum_{\substack{\fkp\in A\\ p\leqslant k}}\frac{[K_{\fkp}:\bQ_p]\log p}{p(p-1)}
	+k\sum_{\substack{\fkp\notin A\\ p\leqslant k}} \frac{[K_{\fkp}:\bQ_p]\log p}{p-1}\\
	&\leqslant C_3(k+D)+k\sum_{\substack{\fkp\in\Specmax\fko_k\setminus A\\ p\leqslant k}} \frac{[K_{\fkp}:\bQ_p]\log p}{p-1},
	\end{align*}
because $\sum_{\fkp\in A}\frac{\log p}{p(p-1)}[K_{\fkp}:\bQ_p]<\infty$. From Definition~\ref{def densite},
	\begin{align*}
	\sum_{\substack{\fkp\in\Specmax\fko_k\setminus A\\ p\leqslant k}} \frac{[K_{\fkp}:\bQ_p]\log p}{p-1}&=(1-\beta_k)\psi_K(k).
	\end{align*}
Since $\beta=\varliminf_k\beta_k$, for every $\varepsilon>0$, there exists $k_0$ such that, for every~$k\geqslant k_0$,
$\beta_k\geqslant \beta-\varepsilon$. 
Moreover, thanks to Lemma~$\ref{poids densite}$, $\psi_K(x)\sim [K:\bQ]\log k$, so there is a non-negative real number~$C_4$ such that, for every $k\in\bZ_{>0}$,
	\[\sum_{\substack{\fkp\in\Specmax\fko_k\setminus A\\ p\leqslant k}} \frac{[K_{\fkp}:\bQ_p]\log p}{p-1}\leqslant C_4+(1-\beta+\varepsilon)[K:\bQ] k\log k.\]
The formal subscheme~$\widehat V$ is therefore $(1-\beta)$-arithmetic.
\end{proof}

\begin{lemme}\label{comparaison densites}
Let~$X$ be a subset of the set of prime numbers. The \emph{natural density} of~$X$ as the real number in~$[0,1]$ is defined by:
	\[d(X)=\varliminf_{N\to\infty} \frac 1 {\pi(N)}\sum_{\substack{2\leqslant n\leqslant N, \\ n\in X}} 1,\]
where $\pi(N)$ denotes the number of prime numbers at most equal to~$N$.

Then the density of~$X$ defined by
	\[\varliminf_{x\to\infty}\frac{1}{\psi_{\bQ}(x)}\sum_{\substack{p\in X,\\  p\leqslant x}}\frac{\log p}{p-1}.\]
is at least equal to the natural density~$d(X)$.
\end{lemme}

\begin{proof}
Let $\mathds 1_X$ denote the characteristic function of~$X$. For every integer~$n\geqslant 2$, we set $a_n=\mathds 1_X(n)$, $b_n=\mathds 1_X(n)\frac{\log n}{n-1}$, and for all $N\geqslant 1$, $A(N)=\sum_{n=2}^Na_n$ and $B(N)=\sum_{n=2}^Nb_n$. Hence $d(X)=\varliminf_N \frac 1{\pi(N)}A(N)$ and the density of~$X$ is $\varliminf_N\frac 1{\psi(N)}B(N)$, where $\psi(N)=\sum_{\substack{n\leqslant N,\\ n \text{ prime}}} \frac{\log n}{n-1}\sim \log(N)$ when $n$ goes to~$\infty$ (see Lemma~\ref{poids densite}).

Let~$\varepsilon>0$, and let $M$ be an integer such that for all $n\geqslant M$, 
	\begin{equation}\label{min A(n)}\frac{1}{\pi(n)}A(n)>d(X)-\varepsilon.\end{equation}
For every $N\geqslant 2$,
	\begin{align*}
	B(N)
	&= \frac{\log N}{N-1}A(N)-\sum_{n=2}^{N-1}A(n)\left(\frac{\log(n+1)}{n}-\frac{\log n}{n-1}\right),
	\end{align*}
by an Abel summation. 
When $N$ goes to~$\infty$,
	\begin{equation}
	\frac 1{\pi(N)}B(N)\geqslant \frac {d(X)-\varepsilon}{\psi(N)}\sum_{n=M}^{N-1}\pi(n)\left(\frac{\log n}{n-1}-\frac{\log(n+1)}{n}\right)+o(1).
	\end{equation}
Since $\frac{\log n}{n-1}-\frac{\log(n+1)}{n}\sim_{n\to\infty}\frac{\log n}{n^2}$ and $\pi(n)\sim \frac n{\log n}$,
	\[\frac 1{\pi(N)}B(N)\geqslant \frac {d(X)-\varepsilon}{\psi(N)}\sum_{n=M}^{N-1}\frac 1 n+\frac {d(X)-\varepsilon}{\psi(N)}\  o\left(\sum_{n=M}^{N-1}\frac 1 n\right)+o(1).\]
Since $\psi(N)\sim\log N$ (Lemma~\ref{poids densite}), letting $\varepsilon$ go to~$0$ we conclude that the density $\varliminf_N \frac 1{\pi(N)}B(N)$ is at least~$d(X)$.	
\end{proof}

We proved in Proposition~\ref{alpha analytique arithmetique} that an $\alpha$-analytic formal subscheme is $\alpha$-arithmetic. The converse does not hold, and here is an example of an $\alpha$-arithmetic formal subscheme which is not $\alpha$-analytic.

Let~$b$ be an algebraic number. Then the formal series~$x$ and~$y$ defined by
\[x(t)=(1+t)^b=\sum_{n=0}^{\infty}\frac{t^n}{n!}b(b-1)\dots(b-n+1),\]
\[y(t)=\frac 1{1+t}=\sum_{n=0}^{\infty}(-t)^n,\]
satisfy the following differential equation:
\[
\left\{
\begin{array}{ccl}
x'(t)&=&bx(t)y(t)   \\
y'(t)&=&-y(t)^2  . 
\end{array}
\right.
\]
Let $K=\bQ(b)$, let~$d$ denote the degree of~$K$ and let $F_b$ be the algebraic foliation over~$K$ (of dimension~1) generated by the vector field
	\[D=bx\frac{\partial}{\partial x}-y\frac{\partial}{\partial y}.\]

\begin{lemme}\label{contre-exemple}
Assume $b$ to be irrational. Then the germ of formal leaf of~$F_b$ through~$(1,1)$ is $\alpha$-analytic if and only if $\alpha\geqslant 1$.
\end{lemme}
\begin{proof}
It follows from Lemma~\ref{equa diff implique 1-analytique}, that the germ of formal leaf of~$F_b$ through~$(1,1)$ is 1-analytic. Let $\alpha\leqslant 1$ and assume that this germ of formal leaf is $\alpha$-arithmetic. Then, by Lemma~\ref{polygone de Newton}, there is a finite subset~$S$ of finite places of~$K$, there is a family $(C_{\fkp})_{\fkp\notin S}$ of real numbers at least~1 with $\prod_{\fkp\notin S}C_{\fkp}<\infty$ such that, for all $r\in\mathopen[0,{C_{\fkp}}^{-1}p^{-\frac{\alpha}{p-1}[K_{\fkp}:\bQ_p]}\mathclose[$, 
	\begin{equation}\label{newton}
	\sup_{t\in\bC_{p} \text{ s.t. } |t|_{\fkp}\leqslant r}|x(t)|_{\fkp}\leqslant C_{\fkp}r.\end{equation}

Let $\fkp$ be such that $b$ is not equal to a rational integer modulo~$\fkp$. Then, for all $r$ less than the radius of convergence of~$x$,
	\begin{align*}
	\sup_{|t|_{\fkp}\leqslant r}|x(t)|_{\fkp}=\max_{n\in\bN}\frac{r^n}{|n!|_{\fkp}}.
	\end{align*}
Thus, for all $r<{C_{\fkp}}^{-1}p^{-\frac{\alpha}{p-1}[K_{\fkp}:\bQ_p]}$,
\[\log\sup_{|t|_{\fkp}\leqslant r}|x(t)|_{\fkp} \geqslant \max_n \left(n\log r+\left(\frac{n}{p-1}-\left\lceil \frac{\log n}{\log p} \right\rceil\right)[K_{\fkp}:\bQ_p]\log p\right).\]
From this inequality and Inequality~\eqref{newton}, letting $r$ go to ${C_{\fkp}}^{-1}p^{-\frac{\alpha}{p-1}[K_{\fkp}:\bQ_p]}$ we obtain that for all $n\in\bZ_{\geqslant 0}$,
	\[\log C_{\fkp}\geqslant \max_n (1-\alpha)\frac{\log p}{p-1}[K_{\fkp}:\bQ_p]+\frac 1 n\left(\frac{\alpha}{p-1}-\left\lceil \frac{\log n}{\log p} \right\rceil\right)[K_{\fkp}:\bQ_p]\log p,\]
and therefore $\log C_{\fkp}\geqslant (1-\alpha)[K_{\fkp}:\bQ_p]\frac{\log p}{p-1}$ for all $\fkp\in\fko_K$ such that $b$ is not a rational integer modulo~$\fkp$. From the Cheborarev theorem, since $b$ is irrational the (natural) density of such $\fkp$ is positive, and thus from Lemma~\ref{comparaison densites} the sum of $\log C_{\fkp}$ over those $\fkp$ diverges unless $\alpha=1$, and the germ of formal leaf of~$F_b$ through~$(1,1)$ is not $\alpha$-arithmetic for $\alpha<1$.
\end{proof}

\begin{lemme}
Assume that the extension $\bQ(b)$ of $\bQ$ is Galois. Then the germ of formal leaf of~$F_b$ through the point~$(1,1)$ is \mbox{$\left(1-\frac 1 {[\bQ(b):\bQ]}\right)$-arithmetic}. 
\end{lemme}
\begin{proof}
From the binomial theorem, for every prime number~$p$,
	\[D^p=\sum_{n=0}^p \binom{p}{n} \left(bx\frac{\partial}{\partial x} \right)^n\left(-y\frac{\partial}{\partial y} \right)^{p-n}.\]
Let~$\fkp$ be a maximal ideal of~$\fko_K$ and let~$p$ be the characteristic of the residue field~$\fko_K/\fkp$. Then $D^p$ is still a derivation after reduction modulo~$\fkp$, and the~$\fkp$-curvature of~$F_b$ vanishes if and only of $F_b$ modulo $\fkp$ is stable under $p$-power. Since
	\begin{align*}D^p&=\left(bx\frac{\partial}{\partial x} \right)^p+\left(-y\frac{\partial}{\partial y} \right)^{p} \mod\fkp\\
	&=b^px\frac{\partial}{\partial x}+(-1)^py\frac{\partial}{\partial y}\mod\fkp,
	\end{align*}
the $\fkp$-curvature of~$F_b$ vanishes if and only if
	\begin{equation}\label{condition sur a}
	b^p=b \mod\fkp.
	\end{equation}
It follows from the Chebotarev density theorem~(\citep{lang_algebraic-number-theory} Theorem~10 page~169) that the natural density of vanishing $\fkp$-curvatures of~$F_b$  is equal to~$\frac 1 {[\bQ(b):\bQ]}$.	Hence, it follows from Lemma~\ref{comparaison densites} that the density defined by~\eqref{densite p-courbures} of vanishing $\fkp$-curvatures is at least equal to~$\frac 1 {[\bQ(b):\bQ]}$, and from Proposition~\ref{densite pc nulle alpha arithmetique}, the leaf of~$F_b$ through~$(1,1)$ is therefore $\left(1-\frac 1 {[\bQ(b):\bQ]}\right)$-arithmetic.
\end{proof}

\section{Formal subschemes based at a closed point} \label{section points fermes}

Until now, we considered formal subschemes of the formal completion of a projective variety at a rational point. In this paragraph, we will define the \emph{$\alpha$-analyticity} of a smooth formal subscheme of the formal completion at any closed point, so as the associated evaluation morphisms, and then we will define what it means for such a formal subscheme to be \emph{$\alpha$-arithmetic}.

Let~$X$ be a projective variety of dimension~$n$ over a number field~$K$ and let~$P$ be a smooth closed point of~$X$. Let~$K(P)$ denote the residue field of~$P$. Then the formal completion~$\widehat X_P$ of~$X$ at~$P$ is isomorphic to the formal spectrum of the ring $K(P)[[t_1,\dots,t_n]]$. Let~$K'$ be a Galois extension of~$K$ containing the residue field~$K(P)$. When we extend scalars from~$K$ to~$K'$, $\Specmax K(P)[[t_1,\dots,t_n]]$ decomposes in a disjoint union of formal schemes~$\Specmax K'[[t_1,\dots,t_n]]$ indexed by the set of embeddings of~$K(P)$ in~$K'$:
	\[\widehat X_P\simeq \bigsqcup_{\sigma:K(P)\hookrightarrow K'} \Specmax K'[[t_1,\dots,t_n]].\]
Let~$\widehat V$ be a smooth formal subscheme of dimension~$d$ of~$\widehat X_P$.  Let~$I_{\widehat V}$  the ideal of $K(P)[[t_1,\dots,t_n]]$ defining~$\widehat V$. When we extend scalars from~$K$ to~$K'$, the formal subscheme~$\widehat V$ decomposes into a disjoint union of $K'$-formal subschemes, indexed by the embeddings of~$K(P)$ in~$K'$:
	\begin{equation}
	\widehat V\otimes_K K'=\bigsqcup_{\sigma:K(P)\hookrightarrow K'}\widehat V_{\sigma},
	\end{equation}
where $\widehat V_{\sigma}$ is defined by the ideal~$I_{\widehat V}\otimes_{K(P),\sigma}K'\subseteq K'[[t_1,\dots,t_n]]$.

\begin{lemme} \label{conjugues a-analytiques}
Let $X$ be a projective variety defined over a number field~$K$ and let~$P$ be a closed point of~$X$. Let~$\widehat V$ be a smooth formal subscheme of~$\widehat X_P$. Let~$K'$ be a Galois extension of~$K$ containing the residue field~$K(P)$ of~$P$. Let~$\alpha$ be a non-negative real number. If there exists an embedding~$\tilde\sigma$ of $K(P)$ in $K'$ such that $\widehat V_{\tilde\sigma}$ is $\alpha$-analytic, then for every embedding~$\sigma:K(P)\hookrightarrow K'$, $\widehat V_{\sigma}$ is $\alpha$-analytic.
\end{lemme}

\begin{proof}We denote by~$n$ the dimension of~$X$ and by~$d$ the dimension of~$\widehat V$. Let $\ol\alpha>\alpha$. Since $\widehat V_{\tilde\sigma}$ is $\alpha$-analytic, it can be parameterized by formal series $f_1,\dots,f_n\in K[[x_1,\dots,x_d]]$, $f_i=\sum_I a_I(i)x^I,$ which, at every place of~$K'$, have a positive radius of convergence and satisfy: there exists a finite subset~$S$ of places~$K'$ such that, for every~$\fkp\in\Specmax(\fko_{K'})\setminus S$, there exists $C_{\fkp}>0$ such that, for every $I\in{\bZ_{\geqslant 0}}^d$, for every $i\in\{1,\dots,n\}$, 

	\[\|a_I(i)\|_{\fkp}\leqslant \frac{C_{\fkp}^{|I|}}{{\|I!\|_{\fkp}}^{\ol\alpha}},\]
	\[\prod_{\fkp\in\Specmax\fko_{K'}}C_{\fkp}<\infty.\]
Let $\sigma$ be an embedding of~$K(P)$ in~$K'$. There exists $\gamma\in\Gal(K'/K(P))$ such that $\sigma=\gamma\tilde\sigma$. The formal subscheme~$\widehat V_{\gamma\tilde\sigma}$ is parameterized by $\gamma\circ f_1,\dots,\gamma\circ f_n$ whose coefficients are the $\gamma(a_I(1)),\dots,\gamma(a_I(n))$. Let $\fkp$ be a maximal ideal of~$\fko_{K'}$. Then
	\begin{align*}
	\|\gamma(a_I(i))\|_{\fkp}&=\|\gamma(a_I(i))\|_{\gamma(\gamma^{-1}(\fkp))}=\|a_I(i)\|_{\gamma^{-1}(\fkp)}
	\leqslant \frac{C_{\gamma^{-1}(\fkp)}^{|I|}}{\|I!\|_{\gamma^{-1}(\fkp)}^{\ol\alpha}}.
	\end{align*}	
As $\gamma_{|\bQ}$ is the identity map, $\gamma^{-1}(\fkp)\cap \bZ=\fkp\cap \bZ$, and hence the integers have the same $\fkp$-adic and $\gamma^{-1}(\fkp)$-adic valuations. So we get
	\[\|\gamma(a_I(i))\|_{\fkp}\leqslant \frac{C_{\gamma^{-1}(\fkp)}^{|I|}}{\|I!\|_{\fkp}^{\ol\alpha}}.\]
Moreover,
	\[\prod_{\fkp\in\Specmax\fko_{K'}}C_{\gamma^{-1}(\fkp)}=\prod_{\fkp\in\Specmax\fko_{K'}}C_{\fkp}<\infty.\]
The formal subscheme $\widehat V_{\sigma}=\widehat V_{\gamma\tilde\sigma}$ is hence $\alpha$-analytic.
\end{proof}

\begin{definition}\label{def alpha-analytique sur K}
Let $X$ be a projective variety defined over a number field~$K$ and let~$P$ be a closed point of~$X$. Let~$\widehat V$ be a smooth formal subscheme of~$\widehat X_P$. Let~$K'$ be a Galois extension of~$K$ containing the residue field~$K(P)$. Let~$\alpha$ be a non-negative real number. The formal subscheme~$\widehat V$ is said to be \emph{$\alpha$-analytic over~$K'$} if for every embedding $\sigma:K(P)\hookrightarrow K'$, $\widehat V_{\sigma}$ is $\alpha$-analytic.
\end{definition}

\begin{lemme}\label{alpha analytique et extension des scalaires}
Let $X$ be a projective variety defined over a number field~$K$ and let~$P$ be a closed point of~$X$. Let~$\widehat V$ be a smooth formal subscheme of~$\widehat X_P$. Let~$K'$ be a Galois extension of~$K$ containing the residue field~$K(P)$ and let~$\alpha$ be a non-negative real number.

If $\widehat V$ is $\alpha$-analytic over~$K'$, then $\widehat V$ is $\alpha$-analytic over~$K''$ for every finite extension~$K''$ of~$K'$.
\end{lemme}

\begin{proof} Let $\ol\alpha>\alpha$. Let $\sigma$ be an embedding of~$K(P)$ in~$K'$ and let $a_I(i)\in K'$ be the coefficients of a parametrization of~$\widehat V_{\sigma}$. If $\widehat V$ is $\alpha$-analytic over~$K'$, there exists a finite subset~$S$ of places of~$K'$ such that, for every $\fkp\in\Specmax(\fko_{K'})\setminus S$,
\[\|a_I(i)\|_{\fkp}\leqslant \frac{C_{\fkp}^{|I|}}{\|I!\|_{\fkp}^{\ol\alpha}}.\]
If we denote by $i$ the inclusion of~$K'$ in~$K''$, then $i\circ\sigma$ is an embedding of~$K(P)$ in~$K''$ and the $a_I(i)$ are the coefficients of a parametrization of $\widehat V_{i\circ\sigma}$. By Lemma~$\ref{conjugues a-analytiques}$, it is sufficient to prove that $\widehat V_{i\circ\sigma}$ is $\alpha$-analytic over~$K''$. Let~$S'$ be the set of the maximal ideals of~$K''$ lying above the maximal ideals of~$K'$ which are in~$S$. Let~$\fkq\in\Specmax(\fko_K)\setminus S'$ and let~$\fkp$ be the maximal ideal of~$K'$ lying under~$\fkq$. Then $\|a_I(i)\|_{\fkq}=\|a_I(i)\|_{\fkp}^{[K''_{\fkq}:K'_{\fkp}]}$. Since $\prod_{\fkq\notin S'}C_{\fkp}^{[K''_{\fkq}:K'_{\fkp}]}={\left(\prod_{\fkp\notin S}C_{\fkp}\right)}^{[K:\bQ]}$ converges, $\widehat V_{i\circ\sigma}$ is $\alpha$-analytic over~$K''$.
\end{proof}

\begin{definition}\label{def alpha-analytique point fermé}Let $X$ be a projective variety defined over a number field~$K$ and let~$P$ be a closed point of~$X$. Let~$\widehat V$ be a smooth formal subscheme of~$\widehat X_P$. Let~$\alpha$ be a non-negative real number. The formal subscheme~$\widehat V$ is said to be \emph{$\alpha$-analytic} if there exists a Galois extension~$K'$ of~$K$, containing the residue field~$K(P)$ of~$P$ such that $\widehat V$ is $\alpha$-analytic over~$K'$.
\end{definition}

We also define a notion of $\alpha$-arithmeticity for such a formal subscheme, based at a closed point. Let~$\sigma$ be an embedding of~$K(P)$ in~$K'$, with $K'$ a Galois extension of~$K$ containing $K(P)$. We recall the definition of the evaluation morphisms along the formal subscheme~$\widehat V_{\sigma}$. For every non-negative integer~$k$, let~$(V_{\sigma})_k$ be the $k$-th infinitesimal neighborhood of~$P^{\sigma}$ in~$\widehat V_{\sigma}$. Hence we have
$\{P^{\sigma}\}=(V_{\sigma})_0,$ for every $k$,  $(V_{\sigma})_k\subseteq (V_{\sigma})_{k+1}$ and $\widehat V_{\sigma}=\lim\limits_{\substack{\rightarrow\\k}}(V_{\sigma})_k.$

Let $L$ be an ample line bundle on~$X$. We define the following $K'$-vector spaces and $K'$-linear maps, for all integers $D,k$:
\[E_{D}=\Gamma(X_{K'},L^{\otimes D}),\]
\begin{align*}
\eta_{D,\widehat V_{\sigma}}:E_{D}&\rightarrow \Gamma(\widehat V_{\sigma},L^D)\\
s&\mapsto s_{|\widehat V_{\sigma}},\end{align*}
\begin{align}
\eta_{D,\widehat V_{\sigma}}^k:E_{D}&\rightarrow \Gamma((V_{\sigma})_k,L^D) \label{eta conjugue}\\
s&\mapsto s_{|(V_{\sigma})_k}.\nonumber
\end{align}
The vector spaces
	\begin{equation}\label{E^k_D conjugue}
	E^k_{D,\widehat V_{\sigma}}=\ker\eta_{D,\widehat V_{\sigma}}^{k-1}=\{s\in\Gamma(X_{K'},L^{\otimes D})\ |s_{|(V_{\sigma})_{k-1}}=0\},
	\end{equation}
for $k\in\bZ_{\geqslant 0}$, define a descending filtration of the vector space~$E_{D}$.

The map $\eta^k_{D,\widehat V_{\sigma}}$ restricted to $E^k_{D,\widehat V_{\sigma}}$ induces a $K'$-linear map
	\begin{equation}\label{morphisme evaluation 1 des conjugues} 
	\varphi^k_{D,\widehat V_{\sigma}}:E^k_{D,\widehat V_{\sigma}}\longrightarrow\Sym^k\left(\Omega^1_{\hat{V_{\sigma}}}\right)\otimes L^D_{P^{\sigma}}.\end{equation}

\begin{lemme}\label{comparaison hauteur morphisme evaluation en les conjugues} Let $k,D$ be non-negative integers. For every place~$v$ of~$K'$, for every embedding~$\sigma_1,\sigma_2$ de~$K(P)$ in~$K'$, we have
 \[h_v(\varphi^k_{D,\widehat V_{\sigma_2}})=h_{\sigma_1\sigma_2^{-1}v}(\varphi^k_{D,\widehat V_{\sigma_1}}).\]
\end{lemme}

\begin{proof} Set $\gamma=\sigma_2\sigma_1^{-1}\in\Gal(K'/K(P))$. The formal subscheme $\widehat V_{\sigma_2}=\widehat V_{\gamma\sigma_1}$ is defined by the ideal~$I_{\widehat V_{\gamma\sigma_1}}=\gamma\left(I_{\widehat V_{\sigma_1}}\right)$. Then
	\begin{equation}\label{morphisme d'evaluation et conjugaison}
	\varphi^k_{D,\widehat V_{\sigma_2}}=\gamma \circ \varphi^k_{D,\widehat V_{\sigma_1}}\circ \gamma^{-1}.
	\end{equation}
Let~$s\in E^k_{D,\widehat V_{\sigma_2}}$, let~$\fkp$ be a maximal ideal of $\fko_{K'}$. Then
	\begin{align*}
	\| \varphi^k_{D,\widehat V_{\sigma_2}}(s)\|_{\fkp}&=\| \gamma \circ \varphi^k_{D,\widehat V_{\sigma_1}}\circ\gamma^{-1}(s))\|_{\fkp}\\
	&=\| \gamma \circ \varphi^k_{D,\widehat V_{\sigma_1}}(\gamma^{-1}(s))\|_{\gamma(\gamma^{-1}\fkp)}\\
	&=\|\varphi^k_{D,\widehat V_{\sigma_1}}(\gamma^{-1}(s))\|_{\gamma^{-1}\fkp}\\
	&\leqslant \| \varphi^k_{D,\widehat V_{\sigma_1}}\|_{\gamma^{-1}\fkp}\ \|s\|_{\fkp}.
	\end{align*}
We proved the following inequality
	\[\|\varphi^k_{D,\widehat V_{\sigma_2}}\|_{\fkp}\leqslant \|\varphi^k_{D,\widehat V_{\sigma_1}}\|_{\gamma^{-1}\fkp}.\]
Applying this inequality to~$\sigma_2$ and $\sigma_1=\gamma^{-1}\sigma_2$ instead of $\sigma_1$ and $\sigma_2$ at the place~$\gamma^{-1}v$ instead of the place~$v$, we get the inequality
	\begin{align*}
	\|\varphi^k_{D,\widehat V_{\sigma_1}}\|_{\gamma^{-1}v}&\leqslant \|\varphi^k_{D,\widehat V_{\sigma_2}}\|_{\gamma(\gamma^{-1}v)} \leqslant \|\varphi^k_{D,\widehat V_{\sigma_2}}\|_v.
	\end{align*}
If~$v$ is an Archimedean place of~$K'$, then
		\begin{align*}
	\| \varphi^k_{D,\widehat V_{\sigma_2}}(s)\|_{v}&=\| \gamma \circ \varphi^k_{D,\widehat V_{\sigma_1}}\circ \gamma^{-1}(s))\|_{v}=\| \varphi^k_{D,\widehat V_{\sigma_1}}(\gamma^{-1}(s))\|_{v\circ\gamma}\\
	&\leqslant \| \varphi^k_{D,\widehat V_{\sigma_1}}\|_{v\circ\gamma}\ \|\gamma^{-1}(s)\|_{v\circ\gamma}
	\leqslant \| \varphi^k_{D,\widehat V_{\sigma_1}}\|_{v\circ\gamma}\ \|s\|_{v}.
	\end{align*}
Hence, $\|\varphi^k_{D,\widehat V_{\sigma_2}}\|_{v}\leqslant \|\varphi^k_{D,\widehat V_{\sigma_1}}\|_{v\circ\gamma}$, and also $\|\varphi^k_{D,\widehat V_{\sigma_1}}\|_{v\circ\gamma}\leqslant \|\varphi^k_{D,\widehat V_{\sigma_2}}\|_{v}$.
\end{proof}

We define the $\alpha$-arithmeticity of~$\widehat V$ as we did for the $\alpha$-analyticity.

\begin{definition}\label{def alpha-arithmetique point ferme}
Let $X$ be a projective variety defined over a number field~$K$ and let~$P$ be a closed point of~$X$. Let~$\widehat V$ be a smooth formal subscheme of the formal completion~$\widehat X_P$. Let~$K'$ be a Galois extension of~$K$ containing the residue field~$K(P)$ of~$P$. Let~$\alpha$ be a non-negative real number. The formal subscheme~$\widehat V$ is said to be \emph{$\alpha$-arithmetic} if for every embedding $\sigma:K(P)\hookrightarrow K'$, $\widehat V_{\sigma}$ is $\alpha$-arithmetic.
\end{definition}

\begin{remarque} The existence of one Galois extension of~$K$ containing~$K(P)$ and one embedding $\sigma$ of $K(P)$ in this extension such that $\widehat V_{\sigma}$ is $\alpha$-arithmetic is a sufficient condition for $\widehat V$ to be $\alpha$-arithmetic. Indeed, it follows from Lemma~\ref{alpha-arithmetique independant du corps} that this definition does not depend on the choice of a Galois extension~$K'$ of~$K$ containing~$K(P)$.

Moreover, et~$\alpha$ be a non-negative real number. It follows from Lemma~\ref{comparaison hauteur morphisme evaluation en les conjugues} that if one of the formal subschemes~$\widehat V_{\sigma}$ is $\alpha$-arithmetic, then they all are. Indeed, let~$\sigma$ be an embedding of~$K(P)$ in~$K'$. Assume that $\widehat V_{\sigma}$ is $\alpha$-arithmetic and let $\ol\alpha$ be a real number bigger than~$\alpha$. By definition, for every finite subset~$S$ of places of~$K'$, there exists a positive real number~$C_{S}$ such that
	\[\frac{1}{[K':\mathbf Q]}\sum_{v\in\Sigma_{K'}\setminus S}\log \left\|\varphi_{D,\sigma}^k\right\|_{v}\leqslant \ol\alpha k\log k+C_S(k+D).\]
Let $\gamma\in\Gal(K'/K(P))$, and~$S$ be a finite subset of places of~$K'$ and let~$k,D$ be two non-negative integers. Then, thanks to Lemma~\ref{comparaison hauteur morphisme evaluation en les conjugues},
	\begin{align*}
	\frac{1}{[K':\mathbf Q]}\sum_{v\in\Sigma_{K'}\setminus S} h_{v}(\varphi_{D,\widehat V_{\gamma\sigma}}^k)&=  \frac{1}{[K':\mathbf Q]}\sum_{v\in\Sigma_{K'}\setminus S} h_{\gamma^{-1}v}(\varphi_{D,\widehat V_{\sigma}}^k)\\
	&= \frac{1}{[K':\mathbf Q]}\sum_{v\in\Sigma_{K'}\setminus \gamma^{-1}S} h_{v}(\varphi_{D,\widehat V_{\sigma}}^k)\\
	&\leqslant \ol\alpha k\log k +C_{\gamma^{-1}S}(k+D),
	\end{align*}
because $\widehat V_{\sigma}$ is $\alpha$-arithmetic. Hence, $\widehat V_{\gamma\sigma}$ is also $\alpha$-arithmetic.

\end{remarque}

\begin{prop}
Let $X$ be a projective variety defined over a number field~$K$ and let~$P$ be a closed point of~$X$. Let~$\widehat V$ be a smooth formal subscheme of~$\widehat X_P$. Let~$\alpha$ be a non-negative real number.
If the formal subscheme~$\widehat V$ is $\alpha$-analytic, then it is $\alpha$-arithmetic.
\end{prop}

\begin{proof}
If $\widehat V$ is $\alpha$-arithmetic, there exists a Galois extension~$K'$ of~$K$ containing~$K(P)$ such that, for every $\sigma:K(P)\hookrightarrow K'$, $\widehat V_{\sigma}$ is $\alpha$-analytic. From Proposition~\ref{alpha analytique arithmetique}, the $\widehat V_{\sigma}$ are also $\alpha$-arithmetic, and $\widehat V$ is by definition $\alpha$-arithmetic.
\end{proof}

\section{Uniformization and order of growth}\label{section uniformization}

Before we give the statement of the main Theorem~\ref{TH SL DEGRES}, we define a notion of \emph{uniformization} of formal subschemes defined at closed points of~$X$ by a holomorphic function on an affine curve~$M_0$ over~$\bC$ or~$\bC_p$. This notion extends that of parametrization by meromorphic functions on the affine line over~$\bC$. If $M$ is the projective compactification of~$M_0$, we also define the \emph{order of growth} of such a holomorphic function on~$M_0$ at every point of~$M\setminus M_0$. This notion generalizes the notion of exponential order of growth of a holomorphic function on~$\bC$.


\subsection{Order of growth}\label{section ordre de croissance}

Let $\fkF$ be a complete, algebraically closed valued field. The cases we will be interested in are $\fkF=\bC$ and $\fkF=\bC_p$.  Let~$X$ be a projective variety over~$\fkF$, $M$ an algebraic curve over~$\fkF$, and~$P$ a point of~$M$. Let~$\ol L$ be a metrized line bundle on~$X$; assume $L$ to be ample.

Let~$T$ be a finite subset of~$M$. The affine curve $M\setminus T$ is a Stein space (see \citep{grauert_remmert_steinspaces, kiehl_nichtarchimedischen}) and therefore there always exists a global section $\Gamma(M\setminus T,\Theta^*(L^{-1}))$ which is not the zero section.

\begin{definition}\label{def ordre de croissance}
Let~$T$ be a non-empty finite subset of~$M(\fkF)$. Let~$\rho=(\rho_{\tau})_{\tau\in T}$ be a family of non-negative real numbers. A holomorphic map $\Theta: M\setminus T\rightarrow X(\fkF)$ is \emph{of order at most $\rho$} with respect to~$T$ if there exists a non-zero global section \mbox{$\eta\in\Gamma(M\setminus T,\Theta^*(L^{-1}))$} such that, for every~$\tau\in T$, if $u_{\tau}$ is a local parameter of~$M$ at~$\tau$, there exist positive real numbers $A_1$, $A_2$ such that
	\begin{equation}\label{ordre de croissance}
	\|\eta(z)\|\leqslant A_1\exp\left(A_2|u_{\tau}(z)|^{-\rho_{\tau}}\right)\ \ \mbox{ for all $z$ sufficiently close to~$\tau$}.
	\end{equation}
\end{definition}

When~$\fkF=\bC$, the map is holomorphic in the sense of complex analytic geometry. When~$\fkF=\bC_p$, it is holomorphic in the sense of rigid analytic geometry (see~\citep{bosch-guntzer-remmert} and \citep{Fresnel-vanderPut}).

\begin{remarque} This definition does not depend on the choice of a local parameter at a point of~$T$. 
\end{remarque}

\begin{lemme}
We use the same notation as in Definition~\ref{def ordre de croissance}. Let $\ol L_1$ and $\ol L_2$ be two ample metrized line bundles on~$X$ and let~$\eta$ be a global section of $\Theta^*L_1^{-1}$ satisfying Inequality~\eqref{ordre de croissance} of the previous definition. Then there exists a non-zero section~$\eta'\in\Gamma(M\setminus T,\Theta^*(L_2^{-1}))$ also satisfying~\eqref{ordre de croissance}.
\end{lemme}

\begin{proof} Let~$N$ be a positive integer such that~$L_1^N\otimes L_2^{-1}$ admits a non-zero global section~$f$ (such an integer does exist  because $L_1$ is ample). Then the section $\eta'=\eta^N\Theta^*f\in\Gamma(M\setminus T,\Theta^*L_2^{-1})$ is suitable.
\end{proof}

\begin{lemme}\label{ordre de croissance d'une section}
Let $\Theta:M\setminus T \rightarrow X(\fkF)$ be holomorphic of order at most~$\rho_{\tau}$ at~$\tau\in T$ and let $P_1,\dots, P_m\in M\setminus T$. We may choose a section~$\eta$ satisfying~\eqref{ordre de croissance} which does not vanish at the points $P_1,\dots,P_m$.
\end{lemme}

\begin{proof} We are going to construct a function~$f$, meromorphic on~$M$, holomorphic on $M\setminus (T\cup \{P_1,\dots, P_m\})$, having poles of order exactly~$n_i$ at~$P_i$ and poles of controlled order at the points of~$T$. Then the section $\tilde{\eta}=f\eta$ does not vanish at $P_1,\dots,P_m$ and is holomorphic at~\mbox{$M\setminus T$}; it will only remain to observe that it still satisfies Inequality~\eqref{ordre de croissance}.

If~$\Delta$ is a divisor on~$M$, denote by~$h^0(\Delta)$ the dimension over~$\fkF$ of the vector space of sections~$H^0(M,\Osheaf_M(\Delta))$. 
Denote by~$K_M$ a canonical divisor on~$M$.

Let $\eta\in\Gamma(M\setminus T,\Theta^*L^{-1})$ satisfying~\eqref{ordre de croissance}. Let $n_i$ be the order of vanishing of~$\eta$ at~$P_i$. Let $n\in\bZ_{\geqslant 0}$ and define $\Delta$ the divisor $\Delta=\sum_{\tau\in T}n[\tau]+\sum n_i[P_i]$. The Riemann-Roch theorem says that
	\[h^0(\Delta)-h^0(K_M-\Delta)=\deg \Delta+1-g, \]
and	 
	\[h^0(\Delta-[P_i])-h^0(K_M-\Delta+[P_i])=\deg \Delta-g,\]
where $g$ is the genus of~$M$. If $\deg \Delta\geqslant 2g$, 
	\[h^0(K_M-\Delta)=h^0(K_M-\Delta+[P_i])=0,\]
and therefore $H^0(M,\Osheaf(\Delta-[P_i]))$ is a hyperplane of $H^0(M,\Osheaf(\Delta)).$ As the field~$\fkF$ is infinite, there exists a function~$f\in H^0(M,\Osheaf(\Delta))$ which does not belong to any of those hyperplanes 
	\[H^0(M,\Osheaf(\Delta-[P_1])),\dots,H^0(M,\Osheaf(\Delta-[P_m])).\]
The function~$f$ is meromorphic on~$M$, holomorphic on $M\setminus (T\cup \{P_1,\dots, P_m\})$ and has poles of order exactly~$n_i$ at~$P_i$ of order at most~$n$ at every~$\tau\in T$.

The section $\tilde{\eta}=f\eta$ does not vanish at $P_1,\dots,P_m$ and is holomorphic on~\mbox{$M\setminus T$}. Let $\tau\in T$ and let~$u_{\tau}$ be a local parameter of~$M$ at~$\tau$. Since $f$ has a pole of order at most~$n$ at~$\tau$, there exist $A_3>0$ and a neighborhood~$U$ of~$\tau$ in~$M$ such that for every~$z\in U$,
$f(z)\leqslant A_3|u_{\tau}(z)|^{-n}$. Hence for every~$z$ sufficiently close to~$\tau$,
	\begin{align*}
	\|\tilde\eta(z)\|&\leqslant A_1A_3\exp\left(A_2|u_{\tau}(z)|^{-\rho_{\tau}}\right)|u_{\tau}(z)|^{-n}.
	\end{align*}
Let~$A_2'>A_2$. Then, when $x$ goes to~$\infty$ ($x$ real), $\exp\left({A_2x^{\rho_{\tau}}}\right)x^{n}=o\left(\exp\left({A_2'x^{\rho_{\tau}}}\right)\right)$ so there exists $A_4>0$ such that
	\begin{align*}
	\|\tilde\eta(z)\|&\leqslant A_4\exp\left(A_2'|u_{\tau}(z)|^{-\rho_{\tau}}\right).
	\end{align*}	
The section $\tilde\eta$ satisfies, like~$\eta$, an inequality which is similar to~\eqref{ordre de croissance} and does not vanish at $P_1,\dots,P_m$.
\end{proof}

\subsection{Uniformization of finite order}\label{enonce SL degres}

Let $X$ be a projective variety over~$\bQ$ and let $x_1,\dots,x_m$ be closed points of~$X$. For every $j\in\{1,\dots,m\}$, denote by~$K_j=\bQ(x_j)$ the residue field of~$x_j$, by~$d_j$ its degree over~$\bQ$ and let~$\widehat V_j$ be a smooth $K_j$-subscheme of dimension~$1$ of the formal completion~$\widehat X_{x_j}$ of~$X$ at~$x_j$. Let~$L$ be an ample line bundle on~$X$. Let~$p_0$ be a given place of~$\bQ$, finite or Archimedean. 

\begin{definition}\label{uniformisation}
The family of formal subschemes $(\widehat V_1,\dots,\widehat V_m)$ admits a \emph{uniformization at the place~$p_0$} if  there exist an affine, smooth, connected curve~$M_0$ over~$\bC_{p_0}$, a holomorphic map
\[\Theta:M_0\rightarrow X(\bC_{p_0})\]
and distinct points $w_1,\dots,w_m$ of~$M_0$ such that $\Theta(w_j) = \xi_j$, where~$\xi_j\in X(\bC_{p_0})$ is a geometric point lying above the closed point~$x_j$, and the germ of formal curve parameterized by~$\Theta$ at~$\xi_j$ coincides with~$\widehat V_j$. 
\end{definition}

\begin{definition}\label{uniformisation ordre}
We denote by~$M$ the smooth projective compactification of~$M_0$ and by~$T$ the (finite) complement of~$M_0$ in~$M$, so that $M_0=M\setminus T$. Let $\rho$ be a non-negative real number. We say that such a uniformization is \emph{of order at most $\rho$} if there exists a family $(\rho_{\tau})_{\tau\in T}$ of non-negative real numbers satisfying $\sum_{\tau}\rho_{\tau}\leqslant \rho$ such that the holomorphic map $\Theta: M\setminus T\rightarrow X(\bC_{p_0})$ is of order at most $(\rho_{\tau})_{\tau}$ with respect to~$T$.
\end{definition}

Assume that the formal subscheme~$\widehat V=\cup_{j=1}^m\widehat V_j$ admits a uniformization at this place~$p_0$. For all $j\in\{1,\dots,m\}$, 
the geometric point~$\xi_j$ defines a morphism from~$\Osheaf_{X,x_j}$ to~$\bC_{p_0}$ whose kernel is the maximal ideal~$\fkm_{x_j}$, and hence which can be factorized to give an embedding~$\sigma_j$ of $\Osheaf_{X,x_j}/\fkm_{x_j}=K_j$ in~$\bC_{p_0}$.
\section{Proof of the main Theorem}

\subsection{Statement}\label{section enonce}

The \emph{Zariski closure} of $\widehat V=\bigcup_{j=1}^m \widehat V_j$ in $X$ is by definition the smallest Zariski-closed subset $Y$ of $X$ (defined over~$\bQ$) such that, for every $j\in\{1,\dots,m\},$ $\widehat V_j\subseteq \widehat Y_{x_j}.$ The formal subscheme $\widehat V$ is said to be \emph{algebraic} if its dimension (here,~1) is equal to the dimension of its closure.

We can now state the following geometrical version of the Schneider-Lang theorem on an affine curve.

\begin{theo}\label{TH SL DEGRES}
Let $X$ be a projective variety defined over~$\bQ$ and let $x_1,\dots,x_m$ be closed points of~$X$. For all $j\in\{1,\dots,m\}$, denote by~$K_j=\bQ(x_j)$ the residue field of~$x_j$ and by~$d_j$ its degree over~$\bQ$. Let $\alpha_1,\dots,\alpha_m$ be non-negative real numbers. For every $j\in\{1,\dots,m\}$, let~$\widehat V_j$ be a smooth $\alpha_j$-arithmetic $K_j$-subscheme of dimension~$1$ of the formal completion~$\widehat X_{x_j}$ of~$X$ at~$x_j$. Assume that the family of formal subschemes $(\widehat V_1,\dots,\widehat V_m)$ admits a uniformization of order at most~$\rho\geqslant 0$ at some finite or Archimedean place~$p_0$ of~$\bQ$.
Let $r$ be the dimension of the Zariski closure of 
~$\widehat V=\bigcup_{j=1}^m \widehat V_j$ in~$X$.

Then, 
\begin{itemize}
\item either $r>1$ and \[\sum_{j=1}^m\frac{1}{\alpha_jd_j}\leqslant \frac{r}{r-1}\rho,\]
\item or $r=1$, that is to say the $\widehat V_j$'s are all algebraic. 
\end{itemize}
\end{theo}


%

\begin{remarque} Under the hypothesis of uniformization of the theorem, if one of the formal subschemes~$\widehat V_j$, $j\in\{1,\dots,m\}$, is algebraic, then they all are. Indeed, if there exists some~$j$ such that $\widehat V_j$ is not Zariski-dense in~$X$, there exists a non-zero rational function~$P$ on~$X$, identically equal to zero on~$\widehat V_j$. The holomorphic function $\Theta^*P$ on $M\setminus T$ vanishes with an infinite order of vanishing at any point~$w_j$ such that $\Theta(w_j)=x_j$. As $M\setminus T$ is connected, $\Theta^*P$ is identically~0. Thus the restriction of~$P$ to the $\widehat V_i$, $i\in\{1,\dots,m\}$ is zero, and none of the formal subschemes~$\widehat V_i$ is Zariski-dense in~$X$.
\end{remarque}

Replacing~$X$ by the Zariski closure of~$\widehat V=\cup_{i=1}^m\widehat V_i$, one can assume that the formal subschemes~$\widehat V_i$ are all dense in~$X$ and that $r$ which is the dimension of the Zariski closure of~$\widehat V$ in~$X$ is equal to~$n$, the dimension of $X$. That is  what we do from now on.


\subsection{Choice of the filtration and evaluation morphisms}\label{morphisme evaluation en plusieurs points}

Let~$L$ be an ample line bundle on~$X$.  We construct a filtration of the space~$E_D$ of the sections of~$L^D$ by the order of vanishing along the formal subschemes~$\widehat V_i$. At a step of the filtration, we do not impose the same oder of vanishing along the~$m$ formal subschemes.

Let $K$ be a finite Galois extension of~$\bQ$, included in~$\bC_{p_0}$ and containing the fields  $\sigma_1(K_1),\dots,\sigma_m(K_m)$, where for every $j\in\{1,\dots,m\}$, $\sigma_j$ is the embedding of~$K_j$ in~$\bC_{p_0}$ given by the uniformization. We define the following $\bQ$-vector spaces and $K$-vector spaces:
	\[E_{\bQ,D}=\Gamma(X,L^D),\]
and
	\[E_{K,D}=E_{\bQ,D}\otimes_{\bQ}K.\]

Let $\eta_{\bQ,D}$ be the restriction morphism from $E_{\bQ,D}=\Gamma(X,L^D)$ to $\bigoplus_{j=1}^m\Gamma(\widehat V_j,L^D)$. We will look at the extension of scalars from $\bQ$ to~$K$ of this $\bQ$-linear map. In order to describe it, we use the following lemma which describes the scalars extension from~$\bQ$ to~$K$ of the $K_j$-vector space $\Gamma(\widehat V_j,L^D)$, for every $j\in\{1,\dots,m\}$.

\begin{lemme}
For every $j\in\{1,\dots,m\}$, let $I_j$ be the definition ideal in $\Osheaf_{\widehat X_{x_j}}\simeq K_j[[U]]$ of~$\widehat V_j$. Then
\[\Gamma(\widehat V_j,L^D)\otimes_{\bQ}K=\bigoplus_{\sigma:K_j\hookrightarrow K}\Gamma(\widehat V_{\sigma(\xi_j)},L^D),\]
where $\widehat V_{\sigma(\xi_j)}$ denote the $K$-formal subsecheme of $\Specf K[[U]]$ defined by the ideal $K\sigma(I_j)$.
\end{lemme}
Hence the scalars extension of~$\eta_{\bQ,D}$ to~$K$ is the $K$-linear map:
	\begin{equation}\label{eta_K}\eta_{K,D}:E_{K,D}\to\bigoplus_{j=1}^m\bigoplus_{\sigma:K_j\hookrightarrow K}\Gamma(\widehat V_{\sigma(\xi_j)},L^D).\end{equation}

\begin{lemme} The following propositions are equivalent:
\begin{enumerate}
\item For all $D$, the map $\eta_{K,D}$ is injective.
\item For all sufficiently big $D$, the map $\eta_{K,D}$ is injective.
\item The formal subschemes~$\widehat V$ are dense in~$X$.
\end{enumerate}
\end{lemme}
\begin{proof}
If $\widehat V$ is not dense, then~$\widehat V$ is included in a hypersurface~$H$ of~$X$. For~$D$ big enough, the ample line bundle~$L^D$ has a non-trivial global section which identically vanishes on~$H$ and therefore $\eta_{K,D}$ is not injective, and Condition~2 implies Condition~3. We now show that 3. implies~1. Assume $\widehat V$ is dense in~$X$ and let~$s$ be a section of~$L^D$ on~$X$ which identically vanishes on~$\widehat V$. Then $\widehat V$ is included in the divisor of~$s$, which is equal to the whole variety~$X$ since the formal subscheme~$\widehat V$ is dense in~$X$. Therefore, $s$ is zero on~$X$ and $\eta_{K,D}$ is injective.
\end{proof}	



Let $(a_k)_{k\in\bZ_{>0}}$ be a sequence of integers between~1 et~$m$ et set~$a_0=0$. For every~$k$, the number~$a_k$ will indicate that we ask the sections in the $(k+1)$-step $E^k_D$ of the filtration to vanish with a bigger order along~$\widehat V_{a_k}$ than at the previous step~$E^{k-1}_D$, and ask no extra vanishing condition along the others formal subschemes.

For every $i\in\{1,\dots,m\}$ and~$k\in\bZ_{\geqslant 0}$, let
	\begin{equation*}\omega_i(k) =\Card\{0\leqslant j<k\ |\ a_j=i\}
	=\sum_{0\leqslant j<k} \delta_{a_j,i},\end{equation*}
where $\delta_{u,v}$ denotes the Kronecker symbol.

We set \begin{equation}\label{def n_k}n_k=\omega_{a_k}(k).\end{equation}

Recall that, for every $j\in\{1,\dots,m\}$, $d_j=[K_j:\bQ]$ and denote by $\sigma_j^1,\dots,\sigma_j^{d_j}$ the~$d_j$ embeddings of~$K_j$ in~$K$.

We define a descending filtration on~$E_{\bQ,D}$ and then on~$E_{K,D}$. First, for all positive integer~$D$ we define the following $\bQ$-vector spaces:
	\[E^0_{\bQ,D}=E_{\bQ,D},\]
and for all positive integer $k$, 
	\[E^k_{\bQ,D}=\{s\in E_{\bQ,D}| \ s_{|(V_{j})_{\omega_j(k)-1}}=0 \mbox{ pour tout }j\in\{1,\dots,m\}\}.\]
And then we define the $K$-vector spaces
	\begin{align}
	E^k_{K,D}&=E^k_{\bQ,D}\otimes_{\bQ}K \nonumber \\
	&=\bigcap_{j=1}^m\bigcap_{\sigma:K_j\hookrightarrow K}\{s\in E_D| \ s_{|(V_{\sigma(\xi_j)})_{\omega_j(k)-1}}=0\} \label{filtration 1 E^k_D sur le grand corps K}\\
	&=\bigcap_{j=1}^m\bigcap_{l=1}^{d_j}\ker\left(\eta^{\omega_j(k)}_{D,\widehat V_{\sigma^l_{a_j}(\xi_j)}}\right),\nonumber
	\end{align}
where $\eta^{\omega_j(k)}_{D,\widehat V_{\sigma^l_{a_j}(\xi_j)}}$ is the map defined by~\eqref{eta conjugue}.

To simplify the notation, we will not write the subscript $K$ anymore  for these \mbox{$K$-vector} spaces. Hence, we set $E_D=E_{K,D}$ and, for every non-negative integer~$k$, $E^k_D=E^k_{K,D}$. We defined a descending filtration of $E_D$, which is separated if $\eta_D$ is injective.
	
The kernel of the restriction map
	\[\bigoplus_{j=1}^m\Gamma\Big({(V_{\sigma(\xi_j)}))}_k,L^D\Big) \to \bigoplus_{j=1}^m\Gamma\Big((V_{\sigma(\xi_j)})\Big)_{k-1},L^D\Big)\]
is isomorphic to $\bigoplus_{j=1}^{m}\Sym^k\left(\Omega^1_{\widehat V_{\sigma(\xi_j)})}\right)\otimes L^D_{\sigma(\xi_j)}$. The map $\eta^k_D$ restricted to~$E^k_D$ induces therefore a linear map
	\begin{equation}\label{morphisme d'evaluation} \varphi^k_D:E^k_D\longrightarrow\bigoplus_{j=1}^{m}\bigoplus_{\sigma:K_j\hookrightarrow K}\Sym^k\left(\Omega^1_{\widehat V_{\sigma(\xi_j)})}\right)\otimes L^D_{|\sigma(\xi_j)},
	\end{equation}
which maps a section of~$L^D$ vanishing at order~$\omega_j(k)$ along the formal subscheme~$\widehat V_j$ for all $j\in\{1,\dots,m\}$ on the $(k+1)$-th ``Taylor coefficients'' of its restrictions to $\widehat V_{\sigma(\xi_1)},\dots,\widehat V_{\sigma(\xi_m)}$. By definition, the kernel of $\varphi_D^k$ is equal to~$E_D^{k+1}.$


We also define a refinement of the previous filtration on $E_D$ obtained by taking the tensor product of the filtration of $E_{\bQ,D}$. On every $E^k_{D}$, we define a new descending filtration:
	\[E^k_D=E^{k,0}_D\supseteq\dots\supseteq E^{k,d_{a_k}-1}_D\supseteq E^{k,d_{a_k}}_D=E^{k+1}_D,\]
where for all $l\in\{1,\dots,d_{a_k}-1\}$, $E^{k,l}_D$ is defined by
\[E^{k,l}_D=\left\{s\in E^{k,l-1}_D| \ s_{\left|\left(V_{\sigma^l_{a_k}(\xi_{a_k})} \right)_{n_k} \right.}=0\right\}.\]
For every non-negative integer~$k$ and every $l\in\{1,\dots,d_{a_k}\}$, let
	\begin{equation}\label{ sigma(xi)}
	\xi_{a_k}^l=\sigma_{a_k}^{l}(\xi_{a_k}).
	\end{equation}
We define the $K$-linear evaluation morphism
	\[\varphi^{k,l}_D:E^{k,l-1}_D\to\Sym^{n_k}\Omega^1_{\widehat V_{\xi_{a_k}^l}}\otimes L^D_{|\xi_{a_k}^l}.\]
The range $\Sym^{n_k}\Omega^1_{\widehat V_{\xi_{a_k}^l}}\otimes L^D_{|\xi_{a_k}^l}$ is a $K$-vector space of dimension~1. The kernel of $\varphi^{k,l}_D$ is $E^{k,l}_D$; let again $\varphi^{k,l}_D$ denote the injective $K$-linear map obtained by taking the quotient:
	\begin{equation}\label{morphisme evaluation double filtration}
	\varphi^{k,l}_D:E^{k,l-1}_D/E^{k,l}_D\hookrightarrow\Sym^{n_k}\Omega^1_{\widehat V_{\xi_{a_k}^l}}\otimes L^D_{|\xi_{a_k}^l}.
	\end{equation}
As its range is of dimension~1, the map~$\varphi^{k,l}_D$ is an isomorphism as soon as $E^{k,l}_D$ is strictly included in~$E^{k,l-1}_D$.
The image $\varphi^{k,l}_D(s)$ of a section $s$ in $E^k_D$ by this evaluation morphism equals to $\varphi^{n_k}_{D,\widehat V_{\sigma^l_{a_k}(\xi_{a_k})}}(s)$ with the notation~\eqref{morphisme evaluation 1 des conjugues} we used for the evaluation morphisms along one of the conjugates of a formal subscheme at a closed point.
\subsection{Slopes inequality}\label{section IP}

For $k\in\bZ_{\geqslant 0}$ and $l\in\{1,\dots,d_{a_k}\}$, the $K$-vector spaces $E_{D}^{k,l}$ et $\Sym^k\Omega^1_{\widehat V_{\sigma_{i}^{l}(\xi_{i})}}$ can be equipped with integral structures thanks to the choice of a projective model of~$X$ over~$\Spec\fko_K$. As explained in Paragraph~\ref{fibres vectoriels hermitiens}, they are equipped with Hermitian structures ( $\overline{\sE_D},\overline{\Omega^1_{\widehat V_{i}}} $), the norm on ~$E^{k,l}_D\otimes\bC$ being the John norm associated with the infinity norm (see page~\pageref{comparaison norme John}). We denote by~$h_J(\varphi^k_D)$ the height of the evaluation morphisms relative to these Hermitian norms.

Then we have the following \emph{slopes inequality} due to J.-B.~Bost (see for instance \citep{bost_bourbaki96, chambert-loir_algebricite, bost_algebraic-leaves, chen_these, bost_slopes}), which reflects the fact that the map 
\[\eta_{K,D}:E_{K,D}\to\bigoplus_{j=1}^m\bigoplus_{\sigma:K_j\hookrightarrow K}\Gamma(\widehat V_{\sigma(\xi_j)},L^D)\]
defined by \eqref{eta_K} is injective:

	\begin{align}\label{IPfiltrationdegre}
	\dega(\overline{\mathscr E_{D}})\leqslant\sum_{k=0}^{\infty}\sum_{l=1}^{d_{a_k}}\rk(E_{D}^{k,l-1}/E_{D}^{k,l}) & \Bigg[\pentemax \left(\Sym^{n_k}\overline{\Omega^1_{\widehat V_{\xi_{a_k}^l}}\otimes\overline{\mathscr L}^D_{|\xi_{a_k}^l}}\right) \nonumber \\
	&+h_J(\varphi_{D}^{k,l})\Bigg],
	\end{align}
where $n_k=\omega_{a_k}(k)$, as defined by~\eqref{def n_k}.	
%
%
%
%
From Inequality~\eqref{majoration hauteur john par norme infinie}, this inequality, which involves the height $h_J(\varphi^k_{K,D})$, also holds for $h(\varphi^k_D)$ even if this height does not come from Hermitian norms at the Archimedean places,
	
	\begin{align}\label{IP plusieurs points}
\dega(\overline{\mathscr E_{D}})\leqslant\sum_{k=0}^{\infty}\sum_{l=1}^{d_{a_k}}\rk(E_{D}^{k,l-1}/E_{D}^{k,l}) & \Bigg[\pentemax \left(\Sym^{n_k}\overline{\Omega^1_{\widehat V_{\xi_{a_k}^l}}\otimes\overline{\mathscr L}^D_{|\xi_{a_k}^l}}\right) \nonumber \\
	&+h(\varphi_{D}^{k,l})\Bigg].
	\end{align}

We give an upper bound for the maximal slope arising in this inequality, so as a lower bound for the arithmetic degree.

\begin{lemme}\label{lemme pente max} With the previous notation, there exists a real number~$C_1>0$ such that
\[\pentemax\left(\Sym^{n_k}\overline{\Omega^1_{\widehat V_{\xi_{a_k}^l}}}\otimes\overline{\mathscr L}^D_{|\xi_{a_k}^l}\right)\leqslant C_1(n_k+D)\leqslant C_1(k+D).\]
\end{lemme}

\begin{proof} 
See~\citep{bost_algebraic-leaves}, Lemmas 4.2 and 4.3.
\end{proof}

The lower bound for the arithmetic degree of~$\overline{\mathscr E}_D$ we will use is a weak form of the \emph{arithmetic Hilbert-Samuel theorem} (see Proposition~4.4 and Lemma~4.1 of~\citep{bost_algebraic-leaves} for the proof): there exists a real number~$C>0$ such that
	\begin{equation}
	\label{Hilbert-Samuel arithmetique}\dega(\overline{\mathscr E}_D)\geqslant -CD^{n+1}.\end{equation}
Hence the slopes inequality~\eqref{IP plusieurs points} yields
\begin{equation}\label{IP avec omega1}
-CD^{n+1}\leqslant \sum_{k=0}^{\infty}\sum_{l=1}^{d_{a_k}}\rk(E_{D}^{k,l-1}/E_{D}^{k,l})(C_1(k+D)+h(\varphi_D^k)).
\end{equation}

\subsection{Choice of the derivation speeds defining the filtration}
\begin{lemme}\label{beta omega}
Let $(\beta_j)_{1\leqslant j\leqslant m}$ be a family of positive rational numbers such that $\sum_{j=1}^m\beta_j=1.$

Then there exists a map $\omega:\bZ_{\geqslant 0}\to\bZ_{\geqslant 0}^m$, $k\mapsto\omega(k)=(\omega_1(k),\dots,\omega_m(k))$, such that
	\begin{enumerate}
	\item $\omega(0)=(0,\dots,0)$;
	\item for every~$k\in\bZ_{>0}$, for every~$i\in\{1,\dots,m\}$, $\omega_{i}(k)-\omega_{i}(k-1)\in\{0,1\}$;
	\item for every~$k\in\bZ_{>0}$, there exists a unique $a_k\in\{1,\dots,m\}$ such that
		\[\omega_{a_k}(k)=\omega_{a_k}(k-1)+1;\] 
	\item for every~$i\in\{1,\dots,m\}$, $\omega_{i}(k)\underset{k\to\infty}{=}\beta_{i}k+O(1)$.
	\end{enumerate}
\end{lemme}

\begin{proof}

Let $\delta$ be a common denominator for the $\beta_j$, that is to say a positive integer such that, for every $j\in\{1,\dots,m\}$, $\delta\beta_j$ is an integer. For every $j\in\{1,\dots,m\}$, we set
\[s_j=\delta\beta_j\in\bZ_{>0}.\]
Then
\[\sum_{j=1}^ms_j=\delta.\]
We define the $\delta$-periodic sequence $(a_k)_{k\in\bZ_{\geqslant 0}}$ whose first $\delta$ terms are:
\[\underbrace{1,\dots,1}_{s_1\text{ times}},\underbrace{2,\dots,2}_{s_2\text{ times}},\dots,\underbrace{m,\dots,m}_{s_m\text{ times}}.\]
Conditions 1, 2 and 3 then define a unique sequence~$\omega$ which is given by:
	\begin{equation}
	\omega\left(j\delta+\sum_{\ell=1}^{c-1}s_{\ell}+l\right)=\Big((j+1)s_1,\dots,(j+1)s_{c-1},js_c+l,js_{c+1},\dots,js_{m}\Big),
	\end{equation}
for $j\in\bZ_{\geqslant 0}$, $c\in\{1,\dots,m\}$ and $l\in\{1,\dots,s_c\}$ (if $c=1$, we consider the sum as empty).

Let~$k\in\bZ_{>0}$ and let~$j=\left\lfloor \frac{k}{\delta} \right\rfloor$. Then $k-\delta <j\delta\leqslant k$, and for every~$i\in\{1,\dots,m\}$ we have:
	\begin{equation*}
	\omega_i(k)\leqslant (j+1)s_i\leqslant \left(\left\lfloor \frac{k}{\delta}\right\rfloor+1\right)\beta_i\delta
	\leqslant \beta_ik+\beta_i\delta.
	\end{equation*}
Moreover,
	\[\omega_i(k)\geqslant js_i\geqslant \left\lfloor \frac{k}{\delta}\right\rfloor\beta_i\delta
	\geqslant \beta_ik-\beta_i\delta.\]
Hence, the sequence $(\omega(k))_k$ satisfies Condition~4: ~for every~$i\in\{1,\dots,m\}$,
	\[\omega_{i}(k)\underset{k\to\infty}{=}\beta_{i}k+O(1).\qedhere\]
\end{proof}

From now on, we make the following hypothesis on the derivation speeds. Let $\beta_1,\dots,\beta_m$ be positive rational numbers whose sum equals~1. We assume that the sequence $\omega(k)$ which describes the derivation speeds is as in Lemma~\ref{beta omega}, namely for every $i\in\{1,\dots,m\}$, the order of vanishing~$\omega_i(k)$ along~$\widehat V_i$ we impose to the elements of~$E^{k}_D$ satisfies
	\begin{equation}\label{hypothese ordre annulation}
	\omega_{i}(k)\underset{k\to\infty}{=}\beta_{i}k+O(1).
	\end{equation}
Then, for every~$i\in\{1,\dots,m\}$ the sequence $(k-\frac{\omega_i(k)}{\beta_i})_{k\in\bZ_{\geqslant 0}}$ is bounded.
Let~$b$ be a non-negative integer which is an upper bound of these sequences. Then we can write, for every $i\in\{1,\dots,m\}$ and every $k\in\bZ_{\geqslant 0}$,
	\begin{equation}\label{ordre de derivation}
	\omega_{i}(k)=\beta_{i}(k-b)+r_i(k),
	\end{equation}
where $(r_i(k))_{k\in\bZ_{\geqslant 0}}$ is a bounded sequence of non-negative integers.

\subsection{Estimation of the height of the evaluation morphisms}\label{hauteur des morphismes}

\begin{prop}\label{maj hauteur}
Let $\lambda$ be a positive real number such that
	\begin{equation}\label{def lambda degres}
	\lambda\rho < 1.
	\end{equation}
Then there exists a real number~$C>0$ such that, for every non-negative integers~$k,D$, every $l\in\{1,\dots,d_{a_k}\}$, and every $\ol{\alpha_{a_k}} >\alpha_{a_k}$,
\[h(\varphi^{k,l}_{D})\leqslant C(k+D)-[K:\bQ]\left(\lambda-\ol{\alpha_{a_k}} d_{a_k}\beta_{a_k}\right)k\log k+[K:\bQ]\lambda k\log D.\]
\end{prop}

In order to prove this proposition which gives a control of the height of the evaluation morphisms~${\varphi^{k,l}_D}$, we prove two lemmas. Lemma~\ref{alpha arithmetique degre} gives an upper bound for the sum of heights of the evaluation morphism at all but finitely many places, unsing the hypothesis of $\alpha$-arithmeticity of the formal subschemes. And Lemma~\ref{bonne maj morphisme d'evaluation} provides a better upper bound at some places of~$K$, thanks to the uniformization of~$\widehat V$.

\begin{lemme}\label{alpha arithmetique degre} For every finite set~$S$ of embeddings of~$K$ in~$\bC_{p_0}$, for all integers~$k\geqslant 0,D\geqslant 1$, for every $l\in\{1,\dots,d_{a_k}\}$, for every~$\ol{\alpha_{a_k}} >\alpha_{a_k}$, there exists a non-negative real number~$C$ such that
	\[\sum_{p\leqslant\infty}\sum_{\substack{\sigma:K\hookrightarrow\bC_{p},\\ \sigma\notin S}}h_{\sigma}(\varphi^{k,l}_D)\leqslant \ol{\alpha_{a_k}}[K:\bQ]n_k\log n_k+C(k+D).\]
\end{lemme}
\begin{proof}
For every place~$v$ of~$K$, the height at~$v$ of the evaluation morphism $\varphi^{k,l}_D$ defined (see~\eqref{morphisme evaluation double filtration}) by	
	\[\varphi^{k,l}_D:E^{k,l-1}_D/E^{k,l}_D\hookrightarrow\Sym^{n_k}\Omega^1_{\widehat V_{\xi_{a_k}^l}\otimes L^D_{|\xi_{a_k}^l}},\] satisfies
	\[h_v(\varphi^{k,l}_D)\leqslant h_v\left(\varphi^{n_k}_{D,\widehat V_{\sigma^l_{a_k}(\xi_{a_k})}}\right),\]
since it is the restriction of the morphism $\varphi^{n_k}_{D,\widehat V_{\sigma^l_{a_k}(\xi_{a_k})}}$ (\ref{morphisme evaluation 1 des conjugues}) to a smaller domain.	
Let $\ol{\alpha_{a_k}} >\alpha_{a_k}$. Let $k\geqslant 0$ and $D>0$ be two integers. Since $\widehat V_{a_k}$ is \mbox{$\alpha$-arithmetic}, for every embedding $\sigma$ of $K_{a_k}$ in~$K$ the formal subscheme $\widehat V_{\sigma(\xi_{a_k})}$ is \mbox{$\alpha_{a_k}$-arithmetic} (see Definition~\ref{def alpha-arithmetique point ferme}), and hence
\[\sum_{p\leqslant\infty}\sum_{\substack{\sigma:K\hookrightarrow\bC_p,\\ \sigma\notin S}}h_{\sigma}(\varphi^{k,l}_D)\leqslant \ol{\alpha_{a_k}}[K:\bQ]n_k\log n_k+C(k+D).\qedhere\]
\end{proof}

By definition, for every embedding $\sigma$ of~$K$ in~$\bC_{p_0}$, the height of the map $\varphi^k_D$ associated with~$\sigma$ is 
$h_{\sigma}(\varphi^{k,l}_{D})=\log\|\varphi^{k,l}_{D}\otimes_{K,\sigma}\bC_{p_0}\|.$

Thanks to the uniformization, for every embedding~$\sigma$ of~$K$ in~$\bC_{p_0}$ such that $\sigma(\xi_{a_k}^l)=\xi_{a_k}$
we get a ``good'' inequality for~$h_{\sigma}(\varphi^{k,l}_{D})$, as stated in the following lemma.

\begin{lemme}\label{bonne maj morphisme d'evaluation}
Let $\lambda$ be a positive real number such that $\lambda\rho < 1$.

There exists a non-negative real number~$C$ such that, for every $k\in\bZ_{\geqslant 0}$, $D\in\bZ_{>0}$ and every~$l\in\{1,\dots,d_{a_k}\}$, for every embedding~$\sigma$ of~$K$ in~$\bC_{p_0}$ such that $\sigma(\xi_{a_k}^l)=\xi_{a_k}$,
	\begin{equation}\label{bonne maj phi^k_D}
	h_{\sigma}(\varphi^{k,l}_D)\leqslant C(k+D)-\lambda k\log \frac{k}{D}.\end{equation}
\end{lemme}


Since the family of formal subschemes $(\widehat V_1,\dots,\widehat V_m)$ admits a uniformization of order at most~$\rho>0$ at~$p_0$, there exists a projective, connected, smooth curve~$M$ over~$\bC_{p_0}$, a finite subset $T\subseteq M$, a holomorphic map
\[\Theta: M\setminus T \rightarrow X(\bC_{p_0}),\]  and distinct points $w_1,\dots,w_m$ of $M\setminus T$ such that $\Theta(w_j)=\xi_j$ and the germ of formal curve parameterized by~$\Theta$ at~$\xi_j$ coincides with~$\widehat V_j$ (\cf ~Definition~\ref{uniformisation}). Since the uniformization is of order at most~$\rho$, by Definition~\ref{def ordre de croissance} there exist a non-zero section $\eta\in\Gamma(M\setminus T,\Theta^*(L^{-1}))$ and a family $(\rho_{\tau})_{\tau}$ of non-negative real numbers such that, if $u_{\tau}$ a local parameter of~$M$ at~$\tau\in T$, there exist positive real numbers $A_1$, $A_2$ such that
	\begin{equation*}
	\|\eta(z)\|\leqslant A_1e^{A_2|u_{\tau}(z)|^{-\rho_{\tau}}}\ \ \text{ for all }z \text{ sufficiently close to } \tau,   	
	\end{equation*}
and $\sum_{\tau\in T}\rho_{\tau}\leqslant\rho.$ By Lemma~\ref{section ordre de croissance}, we can assume that $\eta$ does not vanish at the points $w_1,\dots,w_m$.


\begin{lemme}\label{fonction auxiliaire R ordres differents}
Let $M$ be a projective, smooth, connected algebraic curve over an algebraically closed field~$\fkF$ and let $T,W\subseteq M(\fkF)$ be finite disjoint subsets. For every $\tau\in T$ let~$\mu_t$ be a positive real number. Assume that $\sum_{\tau\in T}\mu_{\tau}<1$.  For every $w\in W$, let~$\beta_w>0$ be such that $\sum_{w\in W}\beta_w=1$. Then for every big enough integer~$a$, there exists a rational function~$R_a$ on~$M$, regular on~$M\setminus W$ with a pole of order exactly equal to~$\lfloor a\beta_w\rfloor$ at $w\in W$ and a zero of order $m_{\tau}\geqslant \lceil a\mu_{\tau}\rceil$ at every $\tau\in T$.
\end{lemme}

\begin{proof}
We consider the divisor with real coefficients $\Delta$ given by
 	\[\Delta=\sum_{w\in W}\beta_w[w]-\sum_{\tau\in T}\mu_{\tau}[\tau].\]
Its degree $\deg(\Delta)=\sum_{w\in W}\beta_w-\sum_{\tau\in T}\mu_{\tau}$ is positive, by hypothesis. If~$D$ is a divisor on~$M$, denote by~$h^0(D)$ the dimension over~$\fkF$ of the space of sections~$H^0(M,\Osheaf_M(D))$. If~$D$ is a divisor with real coefficients, $D=\sum\lambda_P[P],$ with $\lambda_P\in\bR$ for all~$P$, we set $\lfloor D\rfloor$ the divisor with integral coefficients $D=\sum\lfloor \lambda_P\rfloor[P]$. Let~$K_M$ be a canonical divisor on~$M$ and let~$g$ denote de genus of~$M$. By the Riemann-Roch theorem, for every positive integer~$a$, we have :
	\[h^0(\lfloor a\Delta\rfloor)-h^0(K_M-\lfloor a\Delta\rfloor)=\deg(\lfloor a\Delta\rfloor)+1-g,\]
and, for every $w\in W$, 
	\[h^0(\lfloor a\Delta\rfloor -[w])-h^0(K_M-\lfloor a\Delta\rfloor +[w])=\deg(\lfloor a\Delta\rfloor )-g.\]
When $a$ goes to infinity, $\deg(\lfloor a\Delta\rfloor) = a \deg(\Delta)+ O(1)$, and then $\deg(\lfloor a\Delta\rfloor)\sim a \deg(\Delta)$ since $\deg(\Delta)>0$. Pick
$a$ big enough, so that $\deg(\lfloor a\Delta\rfloor)\geqslant 2g$.
Then,
	\[h^0(K_M-\lfloor a\Delta\rfloor)=h^0(K_M-\lfloor a\Delta\rfloor +[w])=0.\]
Therefore, the $\fkF$-vector spaces~$H^0(M,\Osheaf(\lfloor a\Delta\rfloor-[w]))$, for $w\in W$, are hyperplanes of $H^0(M,\Osheaf(\lfloor a\Delta\rfloor ))$. Since the field~$\fkF$ is infinite, there exists 
	\[R_a\in H^0(M,\Osheaf(\lfloor a\Delta\rfloor ))\setminus \bigcup_{w\in W} H^0(M,\Osheaf(\lfloor a\Delta\rfloor -[w])),\] 
which satisfies the conditions we were looking for.	
\end{proof}

\begin{proof}[Proof of Lemma~\ref{bonne maj morphisme d'evaluation}]
To simplify the notation, in this proof we will not indicate by a subscript that the absolute values and norms we consider are those at the place~$p_0$.

Let~$a$ be a positive integer, $(\mu_{\tau})_{\tau\in T}$ a family of positive real number whose sum is less than~$1$ and such that, for every~$\tau\in T$,
	\begin{equation}\label{choix mu_tau degres}
	\mu_{\tau}\geqslant \lambda\rho_{\tau}.
	\end{equation}
This can be done because we assumed $\lambda\rho<1$. 	
If $a$ is big enough, then by Lemma~\ref{fonction auxiliaire R ordres differents}, there exist a rational function~$R_a$ on~$M$, regular on~$M\setminus W$ with a pole of exact order~$\lfloor a\beta_{a_k}\rfloor$ at $w_1,\dots,w_m$ and a zero of order $m_{\tau}\geqslant \lceil a\mu_{\tau}\rceil$ at every $\tau\in T$.

Let $D\in\bZ_{\geqslant 0}$ and let $s\in E_D$.
Let~$f$ be the holomorphic function $\Theta^*(s)\eta^D$ on~\mbox{$M\setminus T$}. Let $k\in\bZ_{\geqslant 0}$ and $l\in\{1,\dots,d_{a_k}\}$ be such that $s\in E^{k,l}_D$. Then $f$ vanishes with order at least~$\omega_i(k)$ at $w_i$, for every $i\in\{1,\dots,m\}$. The image of the section~$s$ by the evaluation morphism~$\varphi^{k,l}_D$ is

	\begin{equation}\label{}
	\varphi_D^{k,l}(s)=c_{k}(\Theta_*\frac{\partial}{\partial z} (w_{a_k}))^{\otimes -n_k}\eta(w_{a_k})^{-D} \in \Sym^{n_k}(\Omega^1_{\widehat V_{\xi_{a_k}}})\otimes L^D_{|\xi_{a_k}},
	\end{equation}
where 
	\[c_{k}= \lim_{z\rightarrow w_{a_k}}\frac{(\Theta^*(s)\eta^D)(z)}{u_{a_k}(z)^{n_k}}
=\lim_{z\rightarrow w_{a_k}}\frac{f(z)}{u_{a_k}(z)^{n_k}},\]
$u_{a_k}$ being a local parameter of the curve~$M$ at~$w_{a_k}$ and \mbox{$n_k=\omega_{a_k}(k)$}, for every non-negative integer~$k$. 

Setting $C_0=\max(|\Theta_*\frac{\partial}{\partial z} (w_{a_k})|^{-1},|\eta(w_{a_k})|)$, we get
	\begin{equation}\label{phi et c_k}
	|\varphi_D^{k,l}(s)|\leqslant C_0^{k+D}|c_{k}|.\end{equation}

Set
	\begin{equation}\label{def nu_ak}
	\nu_{a_k}=a n_k-(k-b)\lfloor a\beta_{a_k}\rfloor.
	\end{equation}
This is a non-negative integer. It is indeed clearly the case if $k<b$, and if $k\geqslant b$ we have:
	\begin{align*}
	\nu_{a_k}=a n_k-(k-b)\lfloor a\beta_{a_k}\rfloor
	&=ar_{a_k}(k)+(k-b)(a\beta_{a_k}-\lfloor a\beta_{a_k}\rfloor)\geqslant 0.
	\end{align*}
Since
\[c_{k}=\lim_{z\rightarrow w_{a_k}}\frac{f(z)}{u_{a_k}(z)^{n_k}},\]
we have
\[|c_{k}|^a=\lim_{z\rightarrow w_{a_k}}\left|{f}^aR_a^{k-b} u_{a_k}^{-\nu_a(k)}(z)\right| \lim_{z\rightarrow w_{a_k}}{\left|R_a^{-1}u_{a_k}^{-\lfloor a\beta_{a_k}\rfloor}(z)\right|}^{k-b}.\]

The function~$R_a$ has a pole of order exactly~$\lfloor a\beta_{j}\rfloor$ at~$w_j$. Setting~$C_1=\max_{1\leqslant j\leqslant m}\lim_{z\rightarrow w_{j}}{\left|R_a(z)^{-1}u_{j}(z)^{-\lfloor a\beta_{j}\rfloor} \right|}^{\frac{1}{a}},$ we thus obtain
	\begin{equation}\label{c_k}|c_{k}|\leqslant C_1^{k-b}{\left(\lim_{z\rightarrow w_{a_k}}\left|{f}^aR_a^{k-b} u_{a_k}^{-\nu_a(k)}(z)\right| \right)}^{\frac 1 a}.\end{equation}

The function $f^aR_a^{k-b}={(\Theta^*(s)\eta^D)}^aR_a^{k-b}$ is holomorphic on~$M\setminus T$, by the hypothesis~\eqref{ordre de derivation} made on the orders of vanishing of~$\Theta^*(s)$ at the points~$w_1,\dots, w_m$. Let~$r$ be a positive real number. We apply to this function a maximum principle on the domain $\{|R_a(z)|\geqslant r^a\}$. If the place~$p_0$ is the Archimedean one, it is the usual maximum principle of complex analysis. If~$p_0$ is a ultrametric place, it is provided by the following proposition, which is proved in~\citep{bost_chambert-loir_analyticcurves} prop B.11.

\begin{prop}\label{principe du max ultrametrique}
Let $\fkF$ be a complete ultrametric field, and let~$M$ be a smooth, connected projective curve on~$\fkF$. Let~$f\in k(M)$ be a non-constant rational function, and let~$X$ be the Weierstraß domain
	\[X=\{x\in M(k) ; |f(x)|\leqslant 1\}.\]
Then, every affinoid function~$g$ on~$X$ ~is bounded. Moreover, there exists $x\in X$ such that \[|g(x)|=\sup_X|g|\ \text{ et }\ |f(x)|=1.\]
\end{prop}

If $r$ is small enough, then for every $i\in\{1,\dots,m\}$, $w_{i}\in\{|R_a(z)|\geqslant r^a\}$. We get:
\begin{eqnarray}
\lim_{z\rightarrow w_{a_k}}\left|{(\Theta^*(s)\eta^D(z))}^aR_a(z)^{k-b}\right|&\leqslant&\max_{\{|R_a(z)|\geqslant r^a\}}\left|{(\Theta^*(s)\eta^D(z))}^aR_a(z)^{k-b}\right| \nonumber\\
&\leqslant&\max_{\{|R_a(z)| = r^a\}}\left|{(\Theta^*(s)\eta^D(z))}^aR_a(z)^{k-b}\right| \nonumber\\
&\leqslant&r^{a(k-b)}\|s\|_{\sigma,\infty}^a\max_{\{|R_a|= r^a\}}{\left|\eta(z)\right|}^{Da}. \label{limite produit fonction aux}
\end{eqnarray}

The section $\eta$ is of order at most~$\rho_{\tau}$ at~$\tau$: by definition (see~\eqref{ordre de croissance}), for any local parameter~$u_{\tau}$ at~$\tau$, there exist positive real numbers~$A_1,A_2$ such that, for all~$z$ close enough to~$\tau$,
\[\|\eta(z)\|\leqslant A_1\exp\left({A_2|u_{\tau}(z)|^{-\rho_{\tau}}}\right).\] 
The function $R_a$ has a zero of order $m_{\tau}$ at $\tau\in T$, so there exists a real number~$A_3>0$ such that for all~$z$ close enough to~$\tau,$ $|R_a(z)|\geqslant A_3 |u_{\tau}(z)|^{m_{\tau}}$. Hence,
	\[\|\eta(z)\|\leqslant A_1\exp\left({A_2A_3^{\rho_{\tau}}|R_a(z)|^{-\frac{\rho_{\tau}}{m_{\tau}}}}\right).\]
For $r$ small enough, we thus obtain
\[\max_{|R(z)|=r^a}{\left\|\eta(z)\right\|}\leqslant \max_{\tau\in T}A_1\exp\left(A_4r^{-a\frac{\rho_{\tau}}{\mu_{\tau}}}\right),\]
where $A_4=A_2\max_{\tau\in T}A_3^{\rho_{\tau}}.$

Recall that $m_{\tau}\geqslant\lceil a\mu_{\tau}\rceil\geqslant a\mu_{\tau}$. Therefore, for $r\leqslant 1$ and $\tau\in T$, we have
\[r^{-\frac{\rho^i_{\tau}}{m_{\tau}}}\leqslant r^{-\frac{\rho^i_{\tau}}{a\mu_{\tau}}}\leqslant r^{-\frac{1}{\lambda a}}.\]
It follows that there exists $r_0\in\mathopen]0,1\mathclose[$ such that, for all $r$ less than~$r_0$,
	\[\max_{|R(z)|=r^a}{\left\|\eta(z)\right\|}\leqslant A_1\exp\left(A_4r^{-\frac{1}{\lambda}}\right).\]

With this bound satisfied by the norm of the section~$\eta$, Inequality~\eqref{limite produit fonction aux} becomes
	\begin{align}
	\lim_{z\rightarrow w_{a_k}}\left|{(\Theta^*(s)\eta^D(z))}^aR_a(z)^{k-b}\right|&\leqslant\max_{|R_a(z)|=r^a}|f^aR_a^{k-b}(z)| \nonumber\\
	&\leqslant r^{a(k-b)}A_1^{aD}\exp\left(A_4aDr^{-\frac{1}{\lambda}}\right)\|s\|_{\sigma,\infty}^a \label{produit fonction aux}.
	\end{align}
By Inequality~\eqref{c_k},
	\begin{eqnarray}
	|c_k|^a
	&\leqslant&C_1^{ak}{\left(\max_{\substack{z\in \mathscr D\text{ et } \\ |u_{a_k}(z)|\leqslant B_1}} \left|{f}^aR_a^{k-b} u_{a_k}^{-\nu_a(k)}(z)\right| \right)},\label{maj c_k}
	\end{eqnarray}
where we denote by $\mathscr D$ a neighborhood of~$w_{a_k}$ on which the local parameter~$u_{a_k}$ is holomorphic, and $B_1$ is a non-negative real number. Applying the maximum principle to the holomorphic function $f^aR_a^{k-b} u_{a_k}^{-\nu_a(k)}$ on the domain
	\[\{|u_{a_k}(z)|\leqslant B_1\}\cap\sD,\] we get:
	\begin{eqnarray*}
	\max_{\substack{z\in \mathscr D\text{ et } \\ |u_{a_k}(z)|\leqslant B_1}} \left|{f}^aR_a^{k-b} u_{a_k}^{-\nu_a(k)}(z)\right|&=&\max_{\substack{z\in \mathscr D\text{ et } \\ |u_{a_k}(z)|= B_1}} \left|{f}^aR_a^{k-b} u_{a_k}^{-\nu_a(k)}(z)\right| \\
	&=&B_1^{-\nu_a(k)}\max_{\substack{z\in \mathscr D\text{ et } \\ |u_{a_k}(z)|= B_1}} \left|{f}^aR_a^{k-b} (z)\right|\\
	&\leqslant&B_1^{-\nu_a(k)}{\left(r^{(k-b)}A_1^{D}\exp\left(A_4Dr^{-\frac{1}{\lambda}}\right)\|s\|_{\sigma,\infty}\right)}^a,\end{eqnarray*}
by Inequality~\eqref{produit fonction aux}. Hence we obtain the following upper bound for~$c_k$, thanks to Inequality~\eqref{maj c_k}:

	\[|c_{k}|\leqslant C_1^{k}B_1^{\frac{-\nu_a(k)}{a}}r^{k-b}A_1^{D}e^{A_4Dr^{-\frac{1}{\lambda}}}\|s\|_{\sigma,\infty}.\]

This upper bound $|c_k|$ enables us to obtain an upper bound for the norm of the image of~$s$ by the evaluation morphism~$\varphi_D^{k,l}$, according to Inequality~\eqref{phi et c_k}:
	\[\|\varphi^{k,l}_D(s)\|_{\sigma}\leqslant C_0^{k+D}C_1^{k}B_1^{\frac{-\nu_a(k)}{a}}r^{k-b}A_1^{D}e^{A_4Dr^{-\frac{1}{\lambda}}}\|s\|_{\sigma,\infty}.\]
If $\sigma$ is an embedding of~$K$ in~$\bC$, that is to say, if $p_0$ is the Archimedean place, since the infinity norm of~$s$ is bounded from above by the associated John norm (see \eqref{comparaison norme John}), we get
	\[\|\varphi^{k,l}_D\|_{\sigma}\leqslant C_0^{k+D}C_1^{k}B_1^{\frac{-\nu_a(k)}{a}}r^{k-b}A_1^{D}e^{A_4Dr^{-\frac{1}{\lambda}}}.\]
Since $\nu_a(k)=O(1)$, there exists a real number~$C_2>0$ such that, for all $r\leqslant r_0$,
	\[h_{\sigma}(\varphi^{k,l}_D)=\log\|\varphi^{k,l}_D\|_{\sigma}\leqslant C_2(k+D)+k\log r+A_4Dr^{-\frac{1}{\lambda}}.\] 
Set
	\[r=\min\left\{r_0,\left(\frac{\lambda k}{A_4D}\right)^{-\lambda}\right\}.\]
If $r=\left(\frac{\lambda k}{A_4D}\right)^{-\lambda}$, \ie for $\frac{k}{D}\geqslant \frac{A_4}{\lambda}r_0^{-\frac{1}{\lambda}}$,
\begin{align}
	\log\|\varphi^{k,l}_D\|_{\sigma}&\leqslant C_2(k+D)-\lambda k\log \left(\frac{\lambda k}{A_4D}\right)+\lambda k\nonumber\\
	&\leqslant C_3(k+D)-\lambda k\log \frac k D,\label{hauteur phi place privilegiee}
\end{align}
where $C_3$ is a positive real number.

If $r=r_0$, that is to say if $\frac{k}{D}\leqslant \frac{A_4}{\lambda}r_0^{-\frac{1}{\lambda}}$, then the norm of~$\varphi^{k,l}_{D}$ satisfies the inequality
	\[\log\|\varphi^{k,l}_{D}\|\leqslant C_0(k+D),\]
which holds at every place by definition of the condition of $\alpha$-arithmeticity.
Yet $\log\frac k D\leqslant -\frac 1{\lambda}\log r_0+\log\frac{A_4}{\lambda}$, so for all~$C_4>0$,
	\begin{align*}
	C_4(k+D)-\lambda k\log\frac k D&\geqslant C_4(k+D)+k\left(\log r_0-\lambda\log\frac{A_4}{\lambda}\right)\\
	&\geqslant(C_4+\log r_0-\lambda\log\frac{A_4}{\lambda})k+C_4D.
	\end{align*}
Choose $C_4=\max(C_0,C_0+\lambda\log\frac{A_4}{\lambda}-\log r_0).$ Then, 
	\[h_{\sigma}(\varphi^{k,l}_D)\leqslant C_0(k+D)\leqslant C_4(k+D)-\lambda k\log \frac k D.\]
Thus Inequality~\eqref{hauteur phi place privilegiee} still holds for small values of  $\frac k D$, replacing $C_3$ by~$C_4$.
Finally, there exists a positive real number~$C$~such that for all non-negative integers~$k$, $D$ and every integer $l\in\{1,\dots,d_{a_k}\}$,
\[h_{\sigma}(\varphi^{k,l}_D)\leqslant C(k+D)-\lambda k\log \frac{k}{D},\]
and this concludes the proof of Lemma~\ref{bonne maj morphisme d'evaluation}.
\end{proof}

\begin{proof}[Proof of Proposition~\ref{maj hauteur}]
By Proposition~\ref{hauteur somme plongements}, the height of the evaluation morphism is given by the following sum:

\begin{align*}
h(\varphi^{k,l}_{D})&=\sum_{p\leqslant\infty}\sum_{\sigma: K\hookrightarrow\bC_p}h_{\sigma}(\varphi^{k,l}_{D})\\
&=\left(\sum_{p\neq p_0}\sum_{\sigma: K\hookrightarrow\bC_p}h_{\sigma}(\varphi^{k,l}_{D})\right)\\
&+\left(\sum_{\substack{\sigma: K\hookrightarrow\bC_{p_0} \\ \sigma\circ\sigma_{a_k}^l(\xi_{a_k})=\xi{a_k}}}h_{\sigma}(\varphi^{k,l}_{D})\right)
+\left(\sum_{\substack{\sigma: K\hookrightarrow\bC_{p_0} \\ \sigma\circ\sigma_{a_k}^l(\xi_{a_k})\neq\xi_{a_k}}}h_{\sigma}(\varphi^{k,l}_{D})\right)
.\end{align*}
Let~$\sigma$ be an embedding of~$K$ in~$\bC_{p_0}.$ If
	\begin{equation}\label{condition bon plongement}
	\sigma(\sigma_{a_k}^l(\xi_{a_k}))=\xi_{a_k},\end{equation}
we can apply Inequality~\eqref{bonne maj phi^k_D} to the height~$h_{\sigma}(\varphi^{k,l}_D)$ according to Lemma~\ref{bonne maj morphisme d'evaluation}.
The composed map $\sigma\circ\sigma^l_{a_k}$ is an embedding of $K_{a_k}$ in~$K$ and is uniquely determined by its image of~$\xi_{a_k}$. For fixed~$k$ and~$l$, the number of embeddings~$\sigma$ of~$K$ in~$\bC_{p_0}$ satisfying Condition~\eqref{condition bon plongement} is equal to the number of different ways of extending an embedding of $\sigma^l_{a_k}(K_{a_k})$ in~$\bC_{p_0}$ to an embedding of~$K$ in~$\bC_{p_0}$, that is
 	\[[K:\sigma^l_{a_k}(K_{a_k})]=[K:K_{a_k}]=\frac{[K:\bQ]}{d_{a_k}}.\]
Hence, denoting by $d$ the degree of~$K$ over~$\bQ$,
\begin{align*}
h(\varphi^{k,l}_{D})&\leqslant C'(k+D)+\ol{\alpha_{a_k}} d n_k\log n_k + \frac{d}{d_{a_k}}(C(k+D)-\lambda k\log \frac{k}{D})\\
&\leqslant C_6(k+D)+\ol{\alpha_{a_k}} d\beta_{a_k}k\log k-\frac{d}{d_{a_k}}\lambda k\log k+\frac{d}{d_{a_k}}\lambda k\log D,
\end{align*}
where $C_6=C'+\frac{dC}{\min_i{d_i}}$, because $n_k\leqslant k$. Hence,
	\[h(\varphi^{k,l}_{D})\leqslant C_6(k+D)-\left(\frac{d\lambda}{d_{a_k}}-d\ol{\alpha_{a_k}} \beta_{a_k}\right)k\log k+\frac{d\lambda}{d_{a_k}} k\log D.\qedhere\]
\end{proof}

\subsection{Proof of the main theorem}\label{dém SL degres}

To prove Theorem~\ref{TH SL DEGRES}, we can assume $n>1$. We will use the slopes inequality with the upper bounds for the height of the evaluation morphisms we proved in Paragraph~\ref{hauteur des morphismes}. We recall the slopes inequality~\eqref{IP avec omega1}

	\[-CD^{n+1}\leqslant\sum_{k=0}^{\infty}\sum_{l=1}^{d_{a_k}}\rk(E_{D}^{k,l-1}/E_{D}^{k,l})\left(C_1(k+D)+h(\varphi_{D}^{k,l})\right).\]
	
	From the upper bound for the height of the evaluation morphism given by Proposition~\ref{maj hauteur}, one has
	\begin{align}
	-C D^{n+1}&\leqslant\sum_{k=0}^{\infty}\sum_{l=1}^{d_{a_k}}\rk(E_{D}^{k,l-1}/E_{D}^{k,l})\Big[C_6(k+D)-\left(\frac{d\lambda}{d_{a_k}}-d\ol{\alpha_{a_k}} \beta_{a_k}\right)k\log k\nonumber \\
	&\qquad\qquad\qquad\qquad\qquad\qquad +\frac{d\lambda}{d_{a_k}} k\log D\Big]\nonumber\\
	&\leqslant \sum_{k=0}^{\infty}\rk(E_{D}^{k}/E_{D}^{k+1})\Big[C_6(k+D)-\left(\frac{d\lambda}{d_{a_k}}-d\ol{\alpha_{a_k}} \beta_{a_k}\right)k\log k\nonumber\\
	&\qquad\qquad\qquad\qquad\qquad\qquad +\frac{d\lambda}{d_{a_k}} k\log D\Big],	
	\label{IP}
	\end{align}
where $C_6$ is a positive real number.

First we prove some inequalities satisfied by the terms of the sequence~$\rk(E^k_D)$. The Hilbert-Samuel theorem provides an estimation of the rank of~$E_D$, since the line bundle~$L$ is ample (see for instance~\citep{bost_algebraic-leaves}~(4.19)) :

\begin{lemme}\label{HS geom}
When $D$ goes to infinity,
	\begin{equation}\label{rang de E}\rk(E_D)\sim \frac{1}{n!}\deg_L(X)D^{n}.\end{equation}
\end{lemme}

\begin{lemme}For all~$k\in\bZ_{\geqslant 0}$ and all $D\in\bZ_{>0}$,
	\begin{equation}\label{difference rangs}\rk(E_D^k/E_D^{k+1})\leqslant d_{a_k}.
	\end{equation}
\end{lemme}
\begin{proof}
This inequality comes from the fact that the map
\[\varphi^{k}_D:E^{k}_D / E^{k+1}_D \to\bigoplus_{l=1}^{d_{a_k}}\Sym^{n_k}\Omega^1_{\widehat V_{\xi_{a_k}^l}}\otimes L^D_{|\xi_{a_k}^l}\]
is injective, and that each $K$-vector space $\Sym^{n_k}\Omega^1_{\widehat V_{\xi_{a_k}^l}}\otimes L^D_{|\xi_{a_k}^l}$ has dimension~$1$.
\end{proof}

\begin{lemme} For every non-negative integer~$N$,
	\begin{equation}\label{asymptotique rang} \rk E^0_D-\rk E_D^N=\sum_{0\leqslant k<N}\rk(E_D^{k}-\rk E_D^{k+1})\leqslant dN,
	\end{equation} that is to say
	\begin{equation}\label{minoration rang} \sum_{k\geqslant N}(\rk E_D^k-\rk E_D^{k+1})=(\rk E_D)^N\geqslant \rk(E^0_D)-dN.
	\end{equation}
\end{lemme}

\begin{proof}For every non-negative integer~$N$,
	\begin{align*}
	\sum_{0\leqslant k<N}(\rk E_D^k-\rk E_D^{k+1})&\leqslant \sum_{0\leqslant k<N}d_{a_k}\text{ from~\eqref{difference rangs}}\\
	&\leqslant \max_{1\leqslant j\leqslant m}{d_j}N,\qedhere
	\end{align*}
\end{proof}

For every non-negative integer~$k$, set $A_k=\frac{d\lambda}{d_{a_k}}$ and $B_k=\frac{d\lambda}{d_{a_k}}-d(\ol{\alpha_{a_k}}\beta_{a_k})$. The slopes inequaliy~\eqref{IP} is then:
	\begin{equation}\label{IP A B}
\sum_{k=0}^{\infty}\rk(E_{D}^{k}/E_{D}^{k+1})\Big[-C_6(k+D)+B_kk\log k-A_k k\log D\Big]\leqslant CD^{n+1}.	
	\end{equation}

We set 
	\begin{equation}\label{A} A=\lambda ,\end{equation} and 
	\begin{equation}\label{B}
	B=\lambda-\max_{k}(\ol{\alpha_{a_k}}\beta_{a_k}d_{a_k})=\lambda-\max_{1\leqslant j\leqslant m}(\ol{\alpha_{j}}\beta_{j}d_j).
	\end{equation}
We will now prove that Inequality~\eqref{IP A B} implies, when $D$ is big enough, 	
	\[(n-1)A \leqslant n(A-B).\]
If $B\leqslant 0$, the conclusion holds (because $A>0$). In all this proof, we assume $0<B$. Then, for every non-negative integer~$k$, $B_k>0$. 	
Let $\beta>0$. We rewrite Inequality~\eqref{IP A B} cutting the sum in two parts, the terms of index $k\leqslant D^{\beta}$ on one side and the terms of index $k>D^{\beta}$ on the other side. Setting
	\[S_D(\beta)=\sum_{k\leqslant D^{\beta}}\rk(E^k_D/E^{k+1}_D)\Big(-C_{6}(k+D)-A_kk\log D+B_kk\log k\Big),\]
and
	\[S'_D(\beta)=\sum_{k>D^{\beta}}\rk(E^k_D/E^{k+1}_D)\Big(-C_{6}(k+D)-A_kk\log D+B_kk\log k\Big),\]
Inequality~\eqref{IP A B} becomes:
	\begin{equation}\label{pentes}S_{D}(\beta)+S'_D(\beta)\leqslant CD^{n+1}.\end{equation}

\begin{lemme}\label{S}
Assume $\beta\geqslant 1$. Then, when $D$ goes to infinity:
	\[|S_D(\beta)|=O(D^{2\beta}\log D).\]
\end{lemme}

\begin{proof}
	\begin{align*}
	|S_D(\beta)|&\leqslant \sum_{k\leqslant D^{\beta}}\rk(E^k_D/E^{k+1}_D)\Big(C_6D+C_6k+A_kk\log D+B_kk\log k\Big)\\
	&\leqslant \sum_{k\leqslant D^{\beta}}\rk(E^k_D/E^{k+1}_D)\Big(C_{6}D+D^{\beta}(C_{6}+A_k\log D)+B_k\beta D^{\beta}\log D\Big)\\
	&\leqslant C_{7}D^{\beta}\log D\sum_{k\leqslant D^{\beta}}\rk(E_D^k/E_D^{k+1}),		\end{align*}
since $\beta\geqslant 1$ and $(A_k)$, $(B_k)$ are bounded. Thus, according to~\eqref{asymptotique rang}, there exists~$C_{8}>0$ such that
	\begin{align*}
	|S_D(\beta)|&\leqslant C_{8}D^{2\beta}\log D.\qedhere
	\end{align*}
\end{proof}
\begin{lemme}\label{S'}
Assume $\beta\in]\frac{A}{B},n[$. Then there exists a positive real number~$C_{11}$ such that, for $D$ big enough, \[S'_D(\beta)\geqslant C_{11}D^{n+\beta}\log D.\]
\end{lemme}
\begin{proof}
We remark that since $\frac{A_k}{B_k}=\frac{\frac{d\lambda}{d_{a_k}}}{\frac{d\lambda}{d_{a_k}}-d(\ol{\alpha_{a_k}}\beta_{a_k})}=\frac{d\lambda}{d\lambda-d(\ol{\alpha_{a_k}}\beta_{a_k}d_{a_k})}\leqslant \frac A B$, then for all non-negative integer~$k$ we have $\beta>\frac{A_k}{B_k}$.

Now, we give a lower bound for $S'_D(\beta)$. If $k\geqslant D^{\beta},$ then
  \begin{align*}
   B_k\log k-A_k\log D&\geqslant(B_k\beta-A_k)\log D\geqslant C_{9}\log D,
  \end{align*}
where $C_{9}= \min_k(B_k\beta-A_k)$ which is positive by the remark.

\[S'_D(\beta)\geqslant\sum_{k>D^{\beta}}\rk(E^k_D/E^{k+1}_D)\Big(-C_{6}(k+D)+C_{9}k\log D)\Big).\]
For $D$ big enough, $-C_{6}+C_{9}\log D\geqslant 0$, and hence
	\begin{align*}
	S'_D(\beta)&\geqslant (-C_{6}D+D^{\beta}(-C_{6}+C_{9}\log D))\sum_{k>D^{\beta}}\rk(E^k_D/E^{k+1}_D)\\
	&\geqslant C_{10}D^{\beta}\log D\sum_{k>D^{\beta}}\rk(E^k_D/E^{k+1}_D),
	\end{align*}
for $D$ big enough.

Yet, $\sum_{k>D^{\beta}}\rk(E^k_D/E^{k+1}_D)\geqslant \rk E^0_D-d([D^{\beta}]+1)$ from~(\ref{minoration rang}). Since $\rk E^0_D\sim CD^r$ thanks to the geometric Hilbert-Samuel theorem (Lemma~\ref{HS geom}) and since $\beta <r$, there exists a positive real number~$C_{10}$ such that, for all big enough $D$,

\[\sum_{k>D^{\beta}}\rk(E^k_D/E^{k+1}_D)\geqslant C_{10}D^n.\]
Setting $C_{11}=C_{9}C_{10}$, we get the lower bound for $S'_D(\beta)$ we wanted, that is: $S'_D(\beta)\geqslant C_{11}D^{n+\beta}\log D.$
\end{proof}

To conclude that $A\leqslant n B$, by contradiction we pick $\beta\in]\frac{A}{B},n[$. Then, by Lemmas~\ref{S} and~\ref{S'},
\[S_D(\beta)+S'_D(\beta)\geqslant -C_{8}D^{2\beta}\log D+C_{11}D^{n+\beta}\log D.\]
Since $2\beta< n+\beta,$ there exists~$C_{12}>0$ such that, for all~$D$ big enough,
\[S_D(\beta)+S'_D(\beta)\geqslant C_{12}D^{n+\beta}\log D.\]
This inequality contradicts Inequality~\eqref{pentes}
\[
S_{D}(\beta)+S'_D(\beta)\leqslant CD^{n+1},\]
because $\beta>\frac{A}{B}\geqslant 1$. Therefore, 
\[\frac{A}{B}\geqslant n.\]
Hence we have
	\[(n-1)A\leqslant n(A-B),\]
that is to say, by definition of~$A$ and~$B$,
	\begin{equation}\label{maj lambda}\lambda\leqslant\frac{n}{n-1}\max_{1\leqslant j\leqslant m}(\ol{\alpha_j}d_j\beta_j).\end{equation}
We recall that the parameter $\lambda$ is a real number which satisfies $\lambda>0$ and $\lambda\rho<1$. If $\rho=0$, then we can chose $\lambda$ as big as we want, and the previous inequality provides a contradiction. Hence $\rho\neq 0$, and letting~$\lambda$ go to~$\frac 1{\rho}$ from below in Inequality~\eqref{maj lambda}
, one gets:
	\begin{equation}\label{maj d_j beta_j}
	1\leqslant \frac{r}{r-1}\rho\max_{1\leqslant j\leqslant m}(\ol{\alpha_j}d_j\beta_j).
	\end{equation}

It remains to make an optimal choice for the rational parameters 
 $\beta_j$. We assume that $\ol{\alpha_j}$ is rational and bigger than $\alpha_j$, for every $j\in\{1,\dots,m\}$. We want to minimize the $\max_{1\leqslant j\leqslant m}\ol{\alpha_j} d_j\beta_j$, for $\beta_j>0$ and $\sum\beta_j=1$. This minimum is at least equal to ${\left(\sum_i\frac{1}{\ol{\alpha_i} d_i}\right)}^{-1}$ since
	\[1=\sum\beta_i=\sum\beta_i \ol{\alpha_i}d_i\frac{1}{\ol{\alpha_i}d_i}\leqslant \max_{j}\ol{\alpha_j}\beta_jd_{j}\sum_i\frac{1}{\ol{\alpha_i}d_i}.\]
Setting
	\begin{equation}\label{choix beta_j}
	\beta_j=\frac{1}{\ol{\alpha_j}d_j}{\left(\sum_i\frac{1}{\ol{\alpha_j}d_i}\right)}^{-1},
	\end{equation}
we have, for every $j\in\{1,\dots,m\}$, $\ol{\alpha_j}d_j\beta_j={\left(\sum_i\frac{1}{\ol{\alpha_i}d_i}\right)}^{-1}.$ With this choice, Inequality~\eqref{maj d_j beta_j} hence becomes $\sum_{i=1}^m\frac{1}{\ol{\alpha_i}d_i}\leqslant \frac{n}{n-1}\rho$.
Letting $\ol{\alpha_j}$ go to~$\alpha_j$ from above with $\ol{\alpha_j}$ rational in the previous inequality, we get
\begin{equation}\sum_{i=1}^m\frac{1}{\alpha_id_i}\leqslant \frac{n}{n-1}\rho,\end{equation}
and this concludes the proof of the main Theorem~\ref{TH SL DEGRES}.	

\begin{remarque}
To prove Theorem~\ref{TH SL DEGRES}, we had to differentiate with different speeds at the different points: at each point, we differentiated with a speed inversely proportional to its degree. If we had differentiated with the same speed at each point, that is to say if we had taken all the~$\beta_j$ equal to~$\frac{1}{m}$, we would have obtained, from~\eqref{maj d_j beta_j}, the following weaker inequality:
\[m\leqslant\frac{r}{r-1}\rho\max_{1\leqslant j\leqslant m}\alpha_jd_j.\]


\end{remarque}


\newpage
\bibliographystyle{apalike}
\bibliography{bibliographie_en.bib}

\end{document}